\let\ORIlabel\label
\let\ORIrefstepcounter\refstepcounter
\AddToHook{package/hyperref/before}{
   \let\label\ORIlabel 
   \let\refstepcounter\ORIrefstepcounter}

\documentclass{siamart220329}
\usepackage{lipsum}
\usepackage{amsfonts}
\usepackage{graphicx}
\usepackage{epstopdf}
\usepackage{algorithmic}
\usepackage{dsfont}
\usepackage{enumitem}
\usepackage{amsmath,amssymb,bm,dsfont}

\setlist[description]{style=unboxed,leftmargin=.5em}
\setlist[itemize]{style=sameline,leftmargin=2em}

\ifpdf  \DeclareGraphicsExtensions{.eps,.pdf,.png,.jpg}
\else
  \DeclareGraphicsExtensions{.eps}
\fi

\newsiamremark{remark}{Remark}
\newsiamremark{hypothesis}{Hypothesis}
\crefname{hypothesis}{Hypothesis}{Hypotheses}
\newsiamthm{claim}{Claim}
\newsiamthm{assumption}{Assumption}

\DeclareMathOperator*{\argmin}{arg\,min}
\newcommand\aff{{\sf aff}}
\newcommand{\boundary}{{\sf bdry}}
\newcommand{\conv}{{\sf conv}}
\newcommand{\dist}{{\sf dist}}
\newcommand{\dom}{{\sf dom}}
\newcommand{\gr}{{\sf gph}}
\renewcommand{\ker}{{\sf ker}}
\newcommand{\interior}{{\sf int}}
\newcommand{\lin}{{\sf lin}}
\newcommand{\range}{{\sf rge}}
\newcommand\Span{{\sf span}}
\renewcommand\Re{{\mathds R}}

\newcommand\rS{{\mathbb S}}
\newcommand\KKT{{\rm KKT}}
\newcommand{\ds}{\displaystyle}
\def\[{\begin{equation}}
\def\]{\end{equation}}
\def\cA{{\mathcal A}}
\def\cB{{\mathcal B}}
\def\cC{{\mathcal C}}
\def\cD{{\mathcal D}}
\def\cE{{\mathcal E}}
\def\cF{{\mathcal F}}
\def\cG{{\mathcal G}}
\def\cH{{\mathcal H}}

\def\cJ{{\mathcal J}}
\def\cK{{\mathcal K}}
\def\cL{{\mathcal L}}
\def\cM{{\mathcal M}}
\def\cN{{\mathcal N}}

\def\cQ{{\mathcal Q}}
\def\cR{{\mathcal R}}
\def\cS{{\mathcal S}}
\def\cT{{\mathcal T}}
\def\cU{{\mathcal U}}
\def\cV{{\mathcal V}}

\def\cX{{\mathcal X}}
\def\cY{{\mathcal Y}}

\def\bfa{{\pmb a}}
\def\bfb{{\pmb b}}
\def\bfc{{\pmb c}}
\def\bfd{{\pmb d}}

\def\bfs{{\pmb s}}
\def\bfu{{\pmb u}}
\def\bfv{{\pmb v}}
\def\bfp{{\pmb p}}
\def\bfq{{\pmb q}}
\def\bfr{{\pmb r}}
\def\bfw{{\pmb w}}
\def\bfx{{\pmb x}}
\def\bfy{{\pmb y}}
\def\bfz{{\pmb z}}
\def\bfzero{{\pmb 0}}

\def\bfnu{{\boldsymbol \nu}}
\def\bfmu{{\boldsymbol\mu}}
\def\bfzeta{{\boldsymbol \zeta}}
\def\bflambda{{\boldsymbol\lambda}}
\def\bfeta{{\boldsymbol\eta}}

\def\bftheta{\boldsymbol{\theta}}
\def\bfxi{\boldsymbol{\xi}}

\headers{Aubin Property and Strong Regularity}{L. Chen, R. N. Chen, D. F. Sun, and J. Y. Zhu}
\title{Aubin Property and Strong Regularity Are Equivalent for Nonlinear Second-Order Cone Programming\thanks{January 7, 2025
\funding{This work was funded by
the National Key R \& D Program of China (No. 2021YFA001300),
the National Natural Science Foundation of China (No. 12271150),
the Hunan Provincial Natural Science Foundation of China (No. 2023JJ10001),
the Science and Technology Innovation Program of Hunan Province (No. 2022RC1190),
the RGC Senior Research Fellow Scheme (No. SRFS2223-5S02),
the GRF Projects (Nos. 15307822 and 15307523),
and the Hunan Provincial Innovation Foundation for Postgraduate (No. CX20220432).}
}}
\author{Liang Chen\thanks{School of Mathematics, Hunan University, Changsha, China
  (\email{chl@hnu.edu.cn}).}
\and Ruoning Chen\thanks{Department of Mathematical Sciences, Tsinghua University, Beijing, China;
and Department of Applied Mathematics, The Hong Kong Polytechnic University, Hung Hom, Hong Kong
(\email{crn22@mails.tsinghua.edu.cn}, \email{ruoning.chen@polyu.edu.hk}).}
\and Defeng Sun\thanks{Department of Applied Mathematics, The Hong Kong Polytechnic University, Hung Hom, Hong Kong
  (\email{defeng.sun@polyu.edu.hk}).}
\and Junyuan Zhu\thanks{School of Mathematics, Hunan University, Changsha, China
  (\email{jyz@hnu.edu.cn}).}}

\usepackage{amsopn}
\ifpdf
\hypersetup{
  pdftitle={Aubin Property and Strong Regularity Are Equivalent for Nonlinear Second-Order Cone Programming},
  pdfauthor={L. Chen, R. N. Chen, D. F. Sun, and J. Y. Zhu}
}
\fi

\begin{document}

\maketitle

\begin{abstract}
This paper solves a fundamental open problem in variational analysis on the equivalence between the Aubin property and the strong regularity for nonlinear second-order cone programming (SOCP) at a locally optimal solution.
We achieve this by introducing a reduction approach to the Aubin property characterized by the Mordukhovich criterion and
a lemma of alternative choices on cones to replace the S-lemma used in Outrata and Ram\'irez [SIAM J. Optim. 21 (2011) 789-823] and Opazo, Outrata, and Ram\'irez [SIAM J. Optim. 27 (2017) 2141-2151], where the same SOCP was considered under the strict complementarity condition except for possibly only one block of constraints.
As a byproduct, we also offer a new approach to the well-known result of Dontchev and Rockafellar [SIAM J. Optim. 6 (1996) 1087-1105] on the equivalence of the two concepts in conventional nonlinear programming.
\end{abstract}

\begin{keywords}
Nonlinear SOCP,
Aubin property,
strong regularity,
strong second-order sufficient condition,
constraint nondegeneracy
\end{keywords}

\begin{MSCcodes}
90C, 90C31, 90C46
\end{MSCcodes}

\section{Introduction}
This paper aims to answer a long-time open problem in variational analysis: whether the \emph{Aubin property} of the solution mapping associated with the canonically perturbed Karush-Kuhn-Tucker (KKT) system and the \emph{strong regularity} of the KKT system are equivalent for nonlinear conic programming without assuming convexity at a locally optimal solution.
In the context of nonlinear SOCP \cite{aliz2003}, here we completely settle this question by proving the equivalence of the two concepts in the absence of strict complementarity.
To be precise, we consider the nonlinear SOCP problem of the form
\[
\label{nlsocp}
\min_{\bfx\in\Re^n} f(\bfx)\quad\text{s.t.}
\quad
h(\bfx)=\bfzero,
\quad
g^j(\bfx)\in\cQ_j,\ j=1,\ldots, J,
\]
where $f:\Re^n\to\Re$, $h:\Re^n\to\Re^m$ and
$g^j:\Re^n\to\Re^{1+r_j}$ are twice continuously differentiable functions,
and $\cQ_j\subseteq\Re^{1+r_j}$ is the second-order cone defined by
$\cQ_j := \left\{{\bfy}\in\Re^{1+r_j}\mid  \dot y\geq \|\bar{\bfy}\| \right\}$.
Here and throughout this paper, the components of vectors in $\Re^{1+r}$ are counted from $0$ to $r$.
For a vector $\bfy\in\Re^{1+r}$, we use $\dot y$ to denote the first component of ${\bfy}$ and define
$\bar {\bfy}:=(y_1,\ldots,y_{r})^\top$, the subvector of $\bfy$ with the first component of $\bfy$ being removed.
Moreover, we write $\bfy=(\dot y;\bar\bfy)$ for simplicity and use
$(\dot g^j(\bfx);\bar g^j(\bfx))$ to represent $g^{j}(\bfx)$ for convenience.
The vectors in $\Re^n$ are indexed in the standard way from $1$ to $n$,
for which $\|\cdot\|$ denotes the Euclidean norm.
For the case that $r_j=0$ in \eqref{nlsocp}, we take the convention that $\cQ_j:=[0,+\infty)$,
so that \eqref{nlsocp} turns to the conventional nonlinear programming if $r_j=0$ for all $1\le j\le J$.
For convenience, we also discuss   \eqref{nlsocp} in the more general conic programming form
\[
\label{op}
\min_{\bfx\in\cX} f(\bfx)
\quad\mbox{s.t.}
\quad
G(\bfx)\in\cK,
\]
where $\cX$ and $\cY$ are two finite-dimensional real Hilbert spaces each endowed with an inner product $\langle \cdot,\cdot\rangle$ and its induced norm $\|\cdot\|$,
$\cK\subseteq\cY$ is a closed convex cone which is $C^2$-cone reducible (c.f. \cite[Definition 3.135]{B&S2000}) at every $\bfy\in\cK$,
and the two functions $f:\cX\to\Re$ and $G:\cX\to\cY$ are twice continuously differentiable.
Given two vectors $\bfa\in\cX$ and $\bfb\in\cY$, the canonically perturbed problem of \eqref{op} is given by
\[
\label{pop}
\min_{\bfx\in\cX} f(\bfx)-\langle \bfa,\bfx\rangle
\quad\mbox{s.t.}
\quad
G(\bfx)-\bfb \in\cK.
\]
Define the Lagrangian function of \eqref{op} by
$$
\cL(\bfx,\bfy)=f(\bfx)-\langle \bfy, G(\bfx)\rangle , \quad   \forall \,(\bfx,\bfy)\in\cX\times\cY.
$$
Then, the KKT system of \eqref{pop} is given by
\[
\label{kktpop}
\bfa=\nabla_{\bfx}\cL(\bfx,\bfy)
\quad\mbox{and}\quad
-\bfy\in\cN_{\cK}(G(\bfx)-\bfb),
\]
where $\nabla_\bfx \cL(\bfx,\bfy)$ denotes the adjoint of  $\cJ_x\cL(\bfx,\bfy)$,
the partial Fr\'echet derivative of $\cL$ with respect to $\bfx$,
and $\cN_{\cK}$ is the normal cone mapping over $\cK$ used in convex analysis.
Moreover, $\bfx^*\in\cX$ is called a stationary point to \eqref{op}, if there exists a multiplier $\bfy^*\in\cY$ such that $(\bfx^*,\bfy^*)$ is a solution to the KKT system \eqref{kktpop} (with $\bfa=\bfzero$ and $\bfb=\bfzero$)
of \eqref{op}.
Then, for the (perturbed) KKT system \eqref{kktpop}, one can define the solution mapping
\[
\label{skkt}
{\rS}_{\KKT}(\bfa,\bfb)
:=
\{(\bfx,\bfy)\mid
\bfa=\nabla_{\bfx}\cL(\bfx,\bfy),\
\bfb\in G(\bfx)+  {\cN_{\cK^*}(\bfy)}
\},
\]
were $\cK^*:=\{ \bfy' \in\cY
\mid \langle \bfy',\bfy \rangle\geq 0, \, \forall \,\bfy\in \cK \}$
is the dual cone of $\cK$.
At a locally optimal solution $\bfx^*$ to \eqref{op} with $\bfy^*$ being the associated multiplier,
i.e., $(\bfx^*,\bfy^*)
\in {\rS}_{\KKT}(\bfzero,\bfzero)$,
this paper is concerned with the following two conditions:

\smallskip
\begin{itemize}
\item[\bf (\romannumeral1)]
\it
The solution mapping ${\rS}_{\rm KKT}$ in \eqref{skkt} has the Aubin property at $(\bfzero,\bfzero)$ for $(\bfx^*,\bfy^*)$, i.e.,
there exists a
constant $\kappa>0$ and open neighborhoods $\cU$ of $(\bfzero,\bfzero)$ and $\cV$ of $(\bfx^*,\bfy^*)$ such that
$$
{\rS}_{\rm KKT}(\bfa',\bfb')\cap\cV\subseteq{\rS}_{\rm KKT}(\bfa,\bfb)+\kappa\|(\bfa',\bfb')-(\bfa,\bfb)\|\mathbb{B}_{\cX\times\cY},\quad
\forall \, (\bfa,\bfb), (\bfa',\bfb')\in\cU,
$$
where $\mathbb{B}_{\cX\times\cY}$ denotes the closed unit ball in ${\cX\times\cY}$ centered at the origin.

\smallskip
\item[\bf (\romannumeral2)]
The solution mapping ${\rS}_{\rm KKT}$ in \eqref{skkt}
has a single-valued Lipschitz continuous localization around
$((\bfzero,\bfzero),(\bfx^*,\bfy^*))$, i.e., there exist a constant $\kappa'>0$ and open neighborhoods $\cU'$ of $(\bfzero,\bfzero)$ and $\cV'$ of $(\bfx^*,\bfy^*)$ such that
${\rS}_{\rm KKT}(\bfa,\bfb)\cap\cV'$ is single-valued for $(\bfa,\bfb)\in\cU'$, and
$${\rS}_{\rm KKT}(\bfa',\bfb')\cap\cV'\subseteq{\rS}_{\rm KKT}(\bfa,\bfb)\cap\cV'+\kappa'\|(\bfa',\bfb')-(\bfa,\bfb)\|\mathbb{B}_{\cX\times\cY},\
\forall \, (\bfa,\bfb), (\bfa',\bfb')\in\cU'.
$$
\end{itemize}
Note that $\mbox{({\bf \romannumeral2 })}\Rightarrow\mbox{({\bf \romannumeral1})}$ trivially holds.
Moreover, according to \cite[Corollary 2.2]{Robinson1980}, $({\bf\romannumeral2})$ is equivalent to the condition that
$(\bfx^*,\bfy^*)$ is a strongly regular solution (in the sense of Robinson \cite{Robinson1980}) of the generalized equation
\[
\label{ge}
0\in \begin{pmatrix}
\nabla_{\bfx}\cL(\bfx,\bfy)
\\
G(\bfx)
\end{pmatrix}
+\begin{pmatrix}
\cN_\cX(\bfx)
\\
 {\cN_{\cK^*}(\bfy)}
\end{pmatrix},
\]
i.e., there exist two neighborhoods $\cU$ of the origin $(\bfzero,\bfzero)\in\cX\times\cY$ and $\cV$ of $(\bfx^*,\bfy^*)$, respectively, such that for every $(\bfa,\bfb)\in \cU$, the linearized generalized equation (c.f. \cite[Definition 22]{bonnans2005} or \cite[Definition 13]{outrata2011})
has a unique solution in $\cV$, denoted by $q_{\cV}(\bfa,\bfb)$, and the mapping $q_{\cV}:\cU\to \cV$ is locally Lipschitz continuous.
Since the generalized equation \eqref{ge} is a reformulation of the KKT system \eqref{kktpop} without perturbation, a strongly regular solution to \eqref{ge} is also called a strongly regular solution to the KKT system of \eqref{op}.

Conditions $({\bf\romannumeral1})$ and $({\bf\romannumeral2})$
belong to central topics in variational analysis, and both of them imply the constraint nondegeneracy, which holds at a feasible solution $\bfx\in\cX$ to   \eqref{op} (or $\bfx\in\cX$ is nondegenerate, for simplicity) if
\[
\label{cnd}
\cJ G(\bfx)\cX+\lin\cT_{\cK}(G(\bfx))=\cY,
\]
where
$\lin \cT_{\cK}(G(\bfx))$ denotes the linearity space, i.e. the largest linear space contained in this tangent cone (in the sense of convex analysis).
The original definition of constrain nondegeneracy takes the form of \cite[Definition 2.1]{Bonans1998}, and
\eqref{cnd} is an equivalent reformulation given in \cite[Definition 16]{bonnans2005} since $\cK$ is $C^2$-reducible to a closed convex cone \cite[Lemma 15]{bonnans2005}.
In particular, it is well-known \cite[Theorem 30]{bonnans2005} that $({\bf\romannumeral2})$ is equivalent to the strong second-order sufficient condition of SOCP   \eqref{nlsocp} at $\bfx^*$  (c.f. \eqref{ssosc}) and the constraint nondegeneracy.
Moreover, these two conditions have played a key role in proving the fast linear or superlinear local convergence rates of augmented Lagrangian methods \cite{chenzhuzhao,liuzhang2008,solodov,nlpCri,sunzhang}.
Additionally, the strong regularity condition also implies the full stability of locally optimal solutions for  \eqref{op}. One may refer to \cite{fullstabsocp,fullstab} for more information.

The most representative case of   \eqref{op} is the conventional nonlinear programming, for which the equivalence of the conditions ({\bf \romannumeral1 }) and ({\bf \romannumeral2 }) has been properly addressed in the seminal work of Dontchev and Rockafellar \cite{don1996}.
Since the proof in \cite{don1996} is highly related to the properties of polyhedral convex sets,
deriving the equivalence of the conditions ({\bf \romannumeral1}) and  ({\bf \romannumeral2}) with a non-polyhedral cone $\cK$ needs new ideas.
An initial important step was made in Outrata and Ram\'irez \cite{outrata2011} (and the erratum by Opazo, Outrata, and Ram\'irez \cite{opazo2017})
in the setting of nonlinear SOCP \eqref{nlsocp},
in which the KKT system \eqref{kktpop} without perturbation was formulated to the generalized equation
$\bfzero\in \nabla f(\bfx)+ \nabla G(\bfx) \cN_\cK (G(\bfx))$ and associated with solution mapping
\[
\label{sge}
{\rS}_{\rm GE}({\bfeta}): =\{\bfx\mid {\bfeta} \in \nabla f(\bfx)+ \nabla G(\bfx) \cN_\cK (G(\bfx))\}.
\]
Under the assumption that the \emph{strict complementarity} condition holds except for possibly only one block of the constraints,
the equivalence of ({\bf \romannumeral1}) and ({\bf \romannumeral2})
was proved in \cite{outrata2011} and \cite{opazo2017}.
As was concluded in \cite[Section 5]{outrata2011},
the authors were uncertain whether such a restriction should be attributed to their proof technique based on the S-lemma,
or the intrinsic property of the second-order cones. 
Note that, however, for convex programming, the equivalence between the two concepts is known from \cite[Proposition 5.1]{dontchev1994}.
In a more general problem stetting, including problem \eqref{op} with $\cK$ being an arbitrary closed convex set,
a recent work \cite[Theorem 4.2]{benko24} shows that the variational sufficient condition for local optimality \cite{rock2023} at $\bfx^*$ is a sufficient condition for the equivalence between ({\bf \romannumeral1}) and ({\bf \romannumeral2}). 
The restrictiveness of this variational sufficient condition will be discussed in Remark \ref{rmkvs}.
Moreover, if additionally $\cK$ is a $C^2$-cone reducible set, 
another recent work \cite[Theorem 5.14]{hang2024} utilized an assumption on relative interiors of subdifferential mappings to prove the equivalence between ({\bf \romannumeral1}) and ({\bf \romannumeral2}), 
and such a condition becomes the strict complementarity condition if \eqref{nlsocp} is the problem in consideration.

In this paper, we aim to prove the equivalence between the two conditions ({\bf i}) and ({\bf ii}) for SOCP \eqref{nlsocp} without assuming either the convexity or the strict complementarity.
To achieve this, we introduce a reduction approach and propose a lemma of alternative choices on cones to supersede the S-lemma used in the existing attempts.
Moreover, as a byproduct, we offer a new approach to the well-known result of Dontchev and Rockafellar \cite{don1996} on the equivalence of  ({\bf \romannumeral1}) and ({\bf \romannumeral2}) for conventional nonlinear programming.

The remaining parts of this paper are organized as follows.
In Section \ref{sec:pre},
we provide some basic definitions and preliminary results in variational analysis.
In Section \ref{sec:corderivative}, we discuss 
coderivatives associated with second-order cones.
In Section \ref{sec:main}, we establish the equivalence between the two conditions ({\bf i}) and ({\bf ii}) by using a reduction approach, in which the lemma of alternative choices on cones developed in Section \ref{sec:pre} plays a key role.
In section \ref{sec:conclu}, we conclude this paper with some discussions.

\section{Notation and preliminaries}
\label{sec:pre}
Let $\cE$ and $\cF$ be two finite-dimensional real Hilbert spaces each equipped with an inner product $\langle \cdot, \cdot \rangle$ and its induced norm $\|\cdot\|$,
and $\Omega$ be a nonempty subset of $\cE$.
We use $\interior\,\Omega$ and $\boundary\,\Omega$ to denote the interior and boundary of $\Omega$, respectively.
For a subspace $\cS\subseteq\cE$, we use $\cS^\perp$ to denote its orthogonal complement.
For a given vector $\bfu \in \cE$, we use $\dist (\bfu,\Omega):=\inf\{\|\bfu'-\bfu\| \mid \bfu'\in\Omega\}$ to define the distance from $\bfu$ to $\Omega$
and use $\Pi_\Omega(\bfu):=\argmin\{\|\bfu'-\bfu\| \mid \bfu'\in\Omega\}$ to denote the projection mapping.
 {For a nonempty cone $K\subseteq\cE$,
$K^\circ := \{ \bfd' \in \cE \mid \langle \bfd',\bfd \rangle\leq 0, \, \forall \,\bfd\in K \}$ is the polar cone of $K$,
and $K^*:=-K^\circ$ is the dual cone of $K$.}
For a linear operator $\cM:\cE\to\cF$,
we use $\range \cM$ and $\ker \cM$ to denote its range space and null space, respectively.
It holds that $\range \cM=(\ker \cM^*)^\perp$, where $\cM^*$ denotes the adjoint of $\cM$.
If $\cE=\cF$, {and $\cM$ is a self-adjoint operator,} we write $\cM\succ O$ to say that $\cM$ is positive definite.
Given a set-valued mapping $ \Psi : \cE \rightrightarrows \cF $,
we use $\dom\Psi$ and $\gr\Psi$ to denote its {domain} and {graph}, respectively.
We use $I_{s}$ to represent the $s\times s$ identity matrix and use
$O_{s\times l}$ (or $O_{s}$ if $s=l$) to represent the $s\times l$ zero matrices,
and we often omit $s$ and $l$ if no ambiguity is introduced.
For convenience, we use $\Span(\cdot )$ to represent the subspace spanned by all the columns of the matrices in the brace.

We present two technical lemmas, which are crucial for our reduction procedure to be introduced later.
The first one is a lemma of alternative choices on cones.

\begin{lemma}
\label{lemtech}
Let $\cM:\cE\to\cE$ be a self-adjoint linear operator and
$Q \subseteq\cE$ be a closed convex cone with $\interior Q \neq\emptyset$.
Assume that  $\langle \bfu, \cM \bfu\rangle > 0$ for all $\bfu\in \boundary  Q\setminus\{\bfzero\}$.
Then exactly one of the following two assertions is true:
\begin{description}
\item[\rm \bf (1)]
$\langle \bfu, \cM\bfu\rangle >0$, $\forall \, \bfu\in\interior Q$.

\item[\rm \bf (2)]
$\exists\,\bfv\in-\interior Q$ and $\lambda\ge 0$ such that $\cM \bfv = -\lambda \bfv \in Q$.
\end{description}
\end{lemma}
\begin{proof}
It is easy to see that Assertion (2) is not true if Assertion (1) holds.
On the other hand, suppose that Assertion (1) does not hold, i.e., $\exists \, \bfu\in\interior Q$
such that $\langle \bfu, \cM\bfu\rangle\le 0$.
Consider the following optimization problem
\[
\label{auxu}
\min\{
\langle \bfu, \cM\bfu\rangle
\mid
\bfu\in\cQ,
\|\bfu\|^2-1= 0\}.
\]
Since the feasible set of \eqref{auxu} is compact and $\langle \bfu, \cM\bfu\rangle >0$, $\forall \, \bfu\in\boundary Q\setminus\{\bfzero\}$, there exists a $\bfu^*\in\interior Q$ that globally solves \eqref{auxu} with $\langle \bfu^*,\cM\bfu^*\rangle\le 0$.
It is also easy to check that Robinson's constraint qualification \cite[Equation (2.182)]{B&S2000} holds at $\bfu^*$ for \eqref{auxu}.
Consequently, there exists a multiplier $(\bfy,\lambda)\in\cE\times\Re$, together with $\bfu^*$, satisfying
$$
2\cM\bfu^*+\bfy+2\lambda\bfu^*=\bfzero,
\quad
\bfu^*\in Q,\,\bfy\in Q^{\circ},\,\langle \bfy,\bfu^*\rangle=0,
\quad
\|\bfu^*\|^2=1.
$$
Since $\bfu^*\in\interior Q$ and $\bfy\in Q^{\circ}$, one must have $\bfy=\bfzero$, so that $\cM\bfu^*=-\lambda\bfu^*$.
Moreover, one has
$\lambda=-\langle \bfu^*,\cM \bfu^*\rangle \ge 0$.
Then, by taking
$\bfv:=-\bfu^*\in -\interior Q$
one has
$\cM\bfv=-\lambda\bfv\in\cQ$.
Thus, Assertion (2) is true.
\end{proof}

The second lemma is a characterization of positive definite matrices that originates from \cite[Theorem 3.6]{han1984}, and one may refer to \cite[Proposition 3.1]{Hiriart} for a more straightforward proof based on the Moreau decomposition.

\begin{lemma}
\label{lemma:pd}
Let $M$ be a nonsingular $l\times l$ real self-adjoint matrix and
$K \subseteq \Re^{l}$ be a closed convex cone.
Then $M\succ O$ if and only if
$$
\begin{cases}
\langle \bfd, M \bfd\rangle>0, & \forall\, \bfd\in K\setminus\{\bfzero\},
\\
\langle \bfd, M^{-1}\bfd\rangle >0, & \forall\, \bfd \in K^\circ\setminus\{\bfzero\}.
\end{cases}
$$
\end{lemma}

In the following, we present some definitions and preliminary results in variational analysis.
The regular (Fr\'{e}chet) normal cone to $\Omega$ at $\bfu\in \Omega$ is defined by
$$\widehat{\cN}_{\Omega}(\bfu) :
= \{ \bfp \in\cE\mid \langle \bfp, \bfu'-\bfu\rangle \leq o(\|\bfu'-\bfu\|), \ \forall \,\bfu'\in \Omega\}.$$
According to \cite[Theorem 1.6]{mord2013}, the limiting (Mordukhovich) normal cone to $\Omega$ at $\bfu$ can be defined by
$$\cN_{\Omega}(\bfu)
:=\mathop{\limsup}\limits_{\bfu'\rightarrow\bfu} \widehat{\cN}_{\Omega}(\bfu'),$$
where $``\limsup"$ denotes the Painlev\'e-Kuratowski outer limit.
When the set $\Omega$ is convex, $\widehat\cN_{\Omega}(\bfu)$ coincides with $\cN_\Omega(\bfu)$, and both of them are called the normal cone
(in convex analysis \cite{rock1970}) to $\Omega$ at $\bfu$.
Based on the definition of limiting normal cone, the {limiting coderivative} of $\Psi$ at $\bfu\in\dom\Psi$ for
$\bfv\in \Psi(\bfu)$ was defined by
\[
\label{def:cod}
\cD^*\Psi(\bfu,\bfv)(\bfq):=\{\bfp\in\cE \mid (\bfp,-\bfq)\in\cN_{\gr \Psi}(\bfu,\bfv)\},\quad\forall\,\bfq\in\cF.
\]
If $\Psi$ is single-valued, we use $\cD^*\Psi(\bfu)$ to represent the limiting coderivative for simplicity.
Define $\Re_+:=[0,+\infty)$ and $\Re_-:=(-\infty, 0]$.
From \eqref{def:cod}, one has for any $v<0$ and any $u>0$ that
\[
\label{coderzerofirst}
\cD^*\cN_{\Re_+}(0,v)(q)
=\begin{cases}
\Re, &\mbox{if } q=0,
\\
\emptyset, &\mbox{otherwise},
\end{cases}
\ \, \mbox{and}\ \
\cD^*\cN_{\Re_+}(u,0)(q)=\{0\},\quad\forall q\in\Re.
\]
Also, for any $v\in\Re$ one has
\[
\label{coderzero}
\cD^*\cN_{\Re_+}(0,0)(q)
=\begin{cases}
\Re_-, & \mbox{ if } q<0 ,
\\
\Re, & \mbox{ if } q=0,
\\
0, & \mbox{ if } q>0,
\end{cases}
\quad \mbox{and}\quad
D^*\cN_{\{ 0\}}(0,v)(q)
=
\begin{cases}
\Re, &\mbox{if } q=0,
\\
\emptyset, &\mbox{otherwise}.
\end{cases}
\]
According to \cite[Lemma 19]{outrata2011},
we have the following result on coderivatives.
\begin{lemma}
\label{codchange}
Let $\Omega \subseteq \cE$ be a closed convex set with $\bfu\in\Omega$ and $\bfv\in \cN_{\Omega}(\bfu)$. Then, one has $\bfp\in\cD^*\cN_{\Omega}(\bfu,\bfv)(\bfq)$ if and only if $
-\bfq\in\cD^*\Pi_{\Omega}(\bfu+\bfv)(-\bfq-\bfp)$.
\end{lemma}

Thanks to Mordukhovich \cite{mord1994}, an
extremely convenient tool for analyzing the Aubin property \cite{aubin1984} via the limiting coderivative is available.
Such a condition was named the Mordukhovich criterion in \cite{varbook}.
It implies the following result.

\begin{lemma}
\label{lemma:aubinorigion}
Let $\bfx^*$ be a stationary point to   \eqref{op}.
Then the mapping ${\rS}_{\rm GE}$ defined by \eqref{sge} has the Aubin property at $(\bfzero,\bfx^*)$
if and only if $\cD^* {\rS}_{\rm GE}(\bfzero, \bfx^*)(\bfzero)=\{\bfzero\}$.
\end{lemma}

Since $\cK$ is $\cC^2$-cone reducible at every $\bfy\in\cK$,
one can use the second-order chain rule developed in \cite[Theorem 7]{outrata2011}, which is a generalization of \cite[Theorem 3.4]{chain}, to get the following result based on repeating the proof of \cite[Theorem 20]{outrata2011}.

\begin{lemma}
\label{lemma:aubin}
Let $\bfx^*$ be a stationary point to   \eqref{op}.
Suppose that the constraint nondegeneracy \eqref{cnd} holds at $\bfx^*$, and
$\bfy^*$ is the corresponding multiplier. Then for the mapping ${\rS}_{\rm GE}$ defined by \eqref{sge} one has
$$
\cD^* {\rS}_{\rm GE}(\bfzero, \bfx^*)(\bfzero)=
\{-\bfd\mid
\bfzero\in \nabla_{\bfx\bfx}^2\cL(\bfx^*,\bfy^*)\bfd
+\nabla G(\bfx^*)\cD^*\cN_\cK(G(\bfx^*),-\bfy^*)(\cJ G(\bfx^*)\bfd)
\}.
$$
\end{lemma}

\section{Coderivatives of the second-order cone}
\label{sec:corderivative}
Let $r\ge 1$ be an integer and
$\cQ:=\{
\bfz=(\dot z; \bar \bfz)
\in\Re^{1+r} \mid
\dot z\ge\|\bar \bfz\|\}$
be a second-order cone.
From \cite[Lemma 15]{aliz2003} we know that for any $\bfz,\bfv\in\cQ$ one has $\langle \bfz, \bfv \rangle=0$ if and only if
either $\bfz=\bfzero$ or $\bfv=\bfzero$, or there exists a scalar  $\alpha>0$ such that $\bfz=\alpha (\dot v;- \bar \bfv)$.
For a given vector $\bfz\in\Re^{1+r}$,
its spectral factorization associated with $\cQ$ is given by $\bfz=\sigma_1(\bfz)\bfc^1(\bfz)+\sigma_2(\bfz)\bfc^2(\bfz)$, where
$$
\bfc^i(\bfz):=
\begin{cases}
\frac{1}{2}\left(
1;(-1)^i\frac{\bar \bfz}{\|\bar \bfz\|}
\right),
&\mbox{ if }\bar \bfz\neq \bfzero,
\\
\frac{1}{2}\left(
1;(-1)^i\bfw
\right),
&\mbox{ if }\bar \bfz=\bfzero,
\end{cases}
\quad
\mbox{and}
\quad
\sigma_i(\bfz)=\dot z+(-1)^i\|\bar \bfz\|,
\quad
i=1,2
$$
with $\bfw\in\mathbb{R}^{r}$ being an arbitrarily vector such that $\|\bfw\|=1$.
Then, according to \cite[Proposition 3.3]{Tseng2001}, the explicit formula of $\Pi_{\cQ}$ can be given as
$$
\Pi_{\cQ}(\bfz)=\max\{0,\sigma_1(\bfz)\}\bfc^1(\bfz)
+\max\{0,\sigma_2(\bfz)\}\bfc^2(\bfz),\quad\forall\,\bfz\in\Re^{1+r}.
$$
It follows from \cite{Zar1971} that $\Pi_{\cQ}$ is almost everywhere differentiable in $ \Re^{1+r} $, and (continuously) differentiable at all $\bfz\in\Re^{1+r}$ such that ${\dot z}^{2} - \|\bar{\bfz}\|^{2} \ne 0$.
The Bouligand subdifferential of $\Pi_{\cQ}$ at $\bfz\in\Re^{1+r}$ is defined by
$$
\partial_{B} \Pi_{\cQ}(\bfz) := \left\{ \lim_{k\rightarrow \infty} \cJ \Pi_{\cQ}(\bfz_k) \, \middle| \, \bfz_{k} \rightarrow \bfz, \Pi_{\cQ} \mbox{ is differentiable at } \bfz_k  \right\}.
$$
Moreover, the explicit formula of $\partial_{B} \Pi_{\cQ}$ has been computed in
\cite[Lemma 14]{pangss},
\cite[Lemma 4]{chenjs}, and \cite[Lemmas 2.5 \& 2.6]{kanzow09}.
In particular, one has
\[
\label{partialbzero}
\partial_{B}\Pi_{\cQ}(\bfzero)=\{I_{1+r}, O_{1+r}\}
\cup
\left\{\frac{1}{2}
\begin{pmatrix} 1 & \bfw^\top \\ \bfw & 2\alpha I_r + (1-2\alpha)\bfw\bfw^\top  \end{pmatrix}
\,\middle| \,
\begin{array}{ll}
\bfw \in\Re^{r},
\|\bfw\|=1,
\\ \alpha \in [0,1]
\end{array} \right\}.
\]
For convenience, define for any $\bfz\in\Re^{1+r}$ (with $\bar{\bfz}\neq \bfzero$) the following two matrices:
\[
\label{defab}
A(\bfz) :=
I_{1+r}-\frac{1}{2}
\begin{pmatrix}
-1
\\
\frac{\bar{\bfz}}{\|\bar{\bfz}\|}
\end{pmatrix}
\begin{pmatrix}
-1  &  \frac{\bar{\bfz}^\top}{\|\bar{\bfz}\|}
\end{pmatrix}
\quad\mbox{and}\quad
B(\bfz) := \frac{1}{2}
\begin{pmatrix}
1
\\
\frac{\bar{\bfz}}{\|\bar{\bfz}\|}
\end{pmatrix}
\begin{pmatrix}
1  &  \frac{\bar{\bfz}^\top}{\|\bar{\bfz}\|}
\end{pmatrix}.
\]
In \cite{O&S2008}, the limiting coderivative $\cD^{\ast} \Pi_{\cQ}$ was explicitly calculated as in the following lemma, in which ``$\conv$'' denotes the convex hull of a set.

\begin{lemma}
\label{projcod}
Let $r\ge 1$ and $\cQ\subseteq\Re^{1+r}$ be a second-order cone and $\bfz, \bfr \in\Re^{1+r}$.

\begin{itemize}
\item [\bf (1)]
If $\bfz\in -\interior\cQ$,
then $\cD^{\ast}\Pi_{\cQ}(\bfz)(\bfr)=\{\bfzero\}$.

\item [\bf (2)]
If $\bfz\in\interior\cQ$, then
$\cD^{\ast}\Pi_{\cQ}(\bfz)(\bfr)=\{\bfr\}$.

\item [\bf (3)]
If $\bfz\notin\cQ\cup(-\cQ)$,
then
$
\cD^{\ast}\Pi_{\cQ}(\bfz)(\bfr)
=\left\{\frac{1}{2}
\begin{pmatrix} 1 & \frac{\bar{\bfz}^\top }{\|\bar{\bfz}\|}
\\
\frac{\bar{\bfz}}{\|\bar{\bfz}\|} & \left(1+\frac{\dot z}{\|\bar{\bfz}\|}\right)I_{r} - \frac{\dot z}{\|\bar{\bfz}\|} \frac{\bar{\bfz}\bar{\bfz}^\top }{\|\bar{\bfz}\|^2}
\end{pmatrix}\bfr
\right\}
$.

\item  [\bf (4)]
If $ \bfz\in\boundary \cQ\setminus\{\bfzero\} $, then
$\cD^{\ast}\Pi_{\cQ}(\bfz)(\bfr)=
\begin{cases}
\conv\{\bfr, A(\bfz)\bfr\}, & \text{if } \langle \bfr, \bfc^1(\bfz) \rangle \ge 0, \\
 \{\bfr, A(\bfz)\bfr\}, & \mbox{ otherwise}.
\end{cases}$

\item [\bf (5)]
If $ \bfz\in\boundary (-\cQ)\backslash\{\bfzero\} $, then
$\cD^{\ast}\Pi_{\cQ}(\bfz)(\bfr)=
\begin{cases}
\conv\{\bfzero, B(\bfz)\bfr\}, & \text{if } \langle \bfr, \bfc^2(\bfz)\rangle \geq 0, \\
\{\bfzero, B(\bfz)\bfr\}, & \text{otherwise}.
\end{cases}$

\item [\bf (6)]
If $\bfz=\bfzero$, then
$$
\cD^{\ast}\Pi_{\cQ}(\bfzero)(\bfr)=
\big\{{\partial}_{B}\Pi_{\cQ}(\bfzero)\bfr \big\}
\cup
\big\{\cQ\cap (\bfr-\cQ)\big\}
\cup\bigcup_{A\in\cA}\conv\{\bfr, A\bfr\} \cup\bigcup_{B\in\cB}\conv\{\bfzero, B\bfr\},
$$
where the two sets $\cA$ and $\cB$ are defined by
$$
\begin{aligned}
\cA := {}&
\left\{
I_{1+r}+\frac{1}{2}  \begin{pmatrix}  -1 & \bfw^\top \\ \bfw & -\bfw\bfw^\top  \end{pmatrix}
\,\middle| \,
\bfw \in\Re^{r}, \|\bfw\|=1,
\langle\bfr, (1; -\bfw) \rangle \geq 0
\right\},
\\
\cB := {}&
\left\{
\frac{1}{2}  \begin{pmatrix}  1 & \bfw^\top \\ \bfw & \bfw\bfw^\top  \end{pmatrix}
\,\middle| \,
\bfw \in\Re^{r}, \|\bfw\|=1,
\langle \bfr, (1; \bfw)\rangle \geq 0 \right\}.
\end{aligned}
$$
\end{itemize}
\end{lemma}

To develop the main result of this paper,
we need the following result on coderivatives of the normal cone mapping to a second-order cone,
which is a combination of Lemmas \ref{codchange} and \ref{projcod}.

\begin{lemma}
\label{codersecond}
Let $r\ge 1$ and $\cQ\subseteq\Re^{1+r}$ be a second-order cone with $\bfu\in\cQ$ and $\bfv\in \cN_{\cQ}(\bfu)$.
\smallskip
\begin{itemize}
\item [\bf (1)]
If $\bfu=\bfzero$ and $\bfv\in-\interior\cQ$, then $\cD^*\cN_{\cQ}(\bfu,\bfv)(\bfzero)=\Re^{1+r}$
and  $\cD^*\cN_{\cQ}(\bfu,\bfv)(\bfq)=\emptyset$, $\forall \, \bfzero \neq \bfq \in \Re^{1+r}$.

\smallskip
\item[\bf (2)]
If $\bfu\in\interior\cQ$ and $\bfv=\bfzero$, then
$\cD^*\cN_{\cQ}(\bfu,\bfv)(\bfq)=\{\bfzero\}$, $\forall \, \bfq\in \Re^{1+r}$.

\item[\bf (3)]
If $\bfu\in\boundary\cQ\setminus\{\bfzero\}$ and $\bfv\in-\boundary\cQ\setminus\{\bfzero\}$,
then for any $\bfq\in\Re^{1+r}$ it holds that
$$
\cD^*\cN_{\cQ}(\bfu,\bfv)(\bfq)
=
\begin{cases}
\left\{
\left(\frac{\dot v}{\dot u}\dot q-\tau;\,
-\frac{\dot v}{\dot u}\bar\bfq+\tau\frac{\bar\bfu}{\|\bar\bfu\|}\right)
\ \middle| \
\tau\in\Re
\right\},
&\mbox{if}\
\langle \bfq,\bfv\rangle=0,
\\
\emptyset,
&\mbox{otherwise}.
\end{cases}
$$

\item[\bf (4)]
If $\bfu\in\boundary\cQ\setminus\{\bfzero\}$ and $\bfv=\bfzero$, then for any  $\bfq\in\Re^{1+r}$ it holds that
$$
\cD^*\cN_{\cQ}(\bfu,\bfv)(\bfq)=\left\{
\tau \left(1 ; -\frac{\bar \bfu}{\|\bar\bfu\|} \right)
\, \middle| \,
\begin{array}{ll}
\tau=0,  &  \mbox{if } \langle (1;-\frac{\bar\bfu}{\|\bar\bfu\|}),\bfq \rangle>0,
\\
\tau\in\Re, &\mbox{if } \langle (1;-\frac{\bar\bfu}{\|\bar\bfu\|}),\bfq \rangle=0,
\\
\tau\le 0, &\mbox{if } \langle (1;-\frac{\bar\bfu}{\|\bar\bfu\|}),\bfq \rangle <0
\end{array}
\right\}.
$$

\item[\bf (5)]
If $\bfu=\bfzero$ and $\bfv\in-\boundary\cQ\setminus\{\bfzero\}$, then for any $\bfq\in\Re^{1+r}$ it holds that
$$
\cD^*\cN_{\cQ}(\bfu,\bfv)(\bfq)
=
\begin{cases}
\left\{ \bfp\in\Re^{1+r} \,\Big | \,
\begin{array}{ll}
\langle \bfp, \bfq \rangle=0, &\mbox{if } \tau>0,
\\
\langle\bfp, \bfq \rangle  \ge 0, &\mbox{if } \tau <0
\end{array}\right\},
& \mbox{if } \bfq=\tau(1;\frac{\bar\bfv}{\|\bar\bfv\|}),
\\
\emptyset, & \mbox{otherwise}.
\end{cases}
$$

\item[\bf (6)]
If $\bfu=\bfv=\bfzero$, then the following conclusions hold:

\begin{itemize}

\item[\bf (6a)]
$\cD^*\cN_{\cQ}(\bfu,\bfv)(\bfzero)=\Re^{1+r}$.

\item[\bf (6b)]
$\bfp\in\cD^*\cN_{\cQ}(\bfu,\bfv)(\bfq)$ holds if $\bfp,\bfq\in-\cQ$.

\item[\bf (6c)]
$\bfp\in\cD^*\cN_{\cQ}(\bfu,\bfv)(\bfq)$ holds if
$\bfp \in -\boundary\cQ\setminus\{\bfzero\}$,  $\bfq \in\Re^{1+r}$
with $\langle \bfp,\bfq \rangle \ge 0$.

\item[\bf (6d)]
$\bfp\in\cD^*\cN_{\cQ}(\bfu,\bfv)(\bfq)$ holds if
$\bfp \in\Re^{1+r}$,
$\bfq \in -\boundary\cQ\setminus\{\bfzero\}$ with $\langle \bfp,\bfq \rangle \ge 0$.

\item[\bf (6e)]
$\bfzero\in\cD^*\cN_{\cQ}(\bfu,\bfv)(\bfq)$, $\forall \,\bfq \in\Re^{1+r}$.

\end{itemize}

\end{itemize}
\end{lemma}

\smallskip
\begin{proof}

{\bf (1)}
Since $\bfu=0$ and $\bfv\in-\interior\cQ$,
one has $\bfu+\bfv\in-\interior\cQ$.
From Part (1) of Lemma \ref{projcod} one can see that
$\cD^{\ast}\Pi_{\cQ}(\bfu+\bfv)(\bfr)=\{\bfzero\}$ for all $\bfr\in\Re^{1+r}$.
By Lemma \ref{codchange},
such a condition implies
$-\bfr \in\cD^*\cN_{\cQ}(\bfu,\bfv)(\bfzero)$
for all $\bfr\in\Re^{1+r}$
so that $\cD^*\cN_{\cQ}(\bfu,\bfv)(\bfzero)=\Re^{1+r}$.
Moreover, if there exists two vectors $\bfp,\bfq\in\Re^{1+r}$ such that $\bfp\in \cD^*\cN_{\cQ}(\bfu,\bfv)(\bfq)$ with $\bfq\neq\bfzero$,
one has from
Lemma \ref{codchange} that $-\bfq\in \cD^{\ast}\Pi_{\cQ}(\bfu+\bfv)(-\bfp-\bfq)$, which constitutes a contradiction to the fact that $\cD^{\ast}\Pi_{\cQ}(\bfu+\bfv)(-\bfp-\bfq)=\{\bfzero\}$.

\smallskip
{\bf (2)}  Since  $\bfu\in\interior\cQ$ and $\bfv=\bfzero$, one has $\bfu+\bfv\in\interior\cQ$.
From Part (2) of Lemma \ref{projcod} one can see that
$\cD^{\ast}\Pi_{\cQ}(\bfu+\bfv)(\bfr)=\{\bfr\}$ for all $ \bfr\in\Re^{1+r}$.
Such a condition is equivalent, by Lemma \ref{codchange}, to the condition that
$\cD^*\cN_{\cQ}(\bfu,\bfv)(-\bfr)=\{\bfzero\}$
for all $\bfr\in\Re^{1+r}$.
Therefore, one has $\cD^*\cN_{\cQ}(\bfu,\bfv)(\bfq)=\{\bfzero\}$ for all $\bfq\in\Re^{1+r}$.

\smallskip
{\bf (3)}
Note that $\bfv\in \cN_{\cQ}(\bfu)$.
Since $\bfu\in\boundary\cQ\setminus\{\bfzero\}$ and
$\bfv\in-\boundary\cQ\setminus\{\bfzero\}$,
one has from \cite[Lemma 15]{aliz2003} that
$-\bfv=-\frac{\dot v}{\dot u}(\dot u;-\bar\bfu)$ and $-\frac{\dot v}{\dot u}>0$.
Therefore, one can get
$\bfu+\bfv=\big((1+\frac{\dot v}{\dot u})\dot u;(1-\frac{\dot v}{\dot u})\bar\bfu\big)$.
It follows that $\bfu+\bfv\notin\cQ\cup(-\cQ)$.
Then by Part $(3)$ of Lemma \ref{projcod} one can get
$\cD^{\ast}\Pi_{\cQ}(\bfu+\bfv)(\bfr)=\{W\bfr\}$
for all $\bfr\in\Re^{1+r}$,
where
$$
\begin{array}{ll}
W=
\\
\frac{1}{2}
\begin{pmatrix}
1&\frac{(\bar\bfu +\bar\bfv)^\top}{\|\bar\bfu +\bar\bfv\|}
\\
\frac{\bar\bfu +\bar\bfv}{\|\bar\bfu +\bar\bfv\|}&
(1+\frac{\dot u +\dot v}{\| \bar\bfu +\bar\bfv\|})
I_{r}-\frac{\dot u+\dot v}{\|\bar\bfu+\bar\bfv\|}
\frac{(\bar\bfu+\bar\bfv)(\bar\bfu+\bar\bfv)^\top}{\|\bar\bfu+\bar\bfv\|^2}
\end{pmatrix}
=
\frac{1}{2}
\begin{pmatrix}
1&\frac{\bar\bfu^\top}{\|\bar\bfu\|} \\
\frac{\bar\bfu}{\|\bar\bfu\|} &
\frac{2\dot u}{{\dot u}-{\dot v}} I_{r}
-\frac{{\dot u}+{\dot v}}{{\dot u}-{\dot v}}
\frac{\bar\bfu\bar\bfu^\top}{\|\bar\bfu\|^2}
\end{pmatrix}.
\end{array}
$$
This condition is equivalent, by Lemma \ref{codchange}, to the following condition:
\[
\label{d3preliminary}
\cD^*\cN_{\cQ}(\bfu,\bfv)(\bfq)=\{\bfp\mid \bfp =-\bfq-\bfr, -W\bfr=\bfq\}
=
\{\bfp\mid W\bfp=(I-W)\bfq\}
,\quad\forall \,\bfq\in\Re^{1+r}.
\]
It is easy to see that $W (\dot u;-\bar\bfu)=\bfzero$, so that $W \bfv
= \frac{\dot v}{\dot u} W (\dot u;-\bar\bfu)=\bfzero$.
Since $W^\top=W$, the condition $W\bfp=(I-W)\bfq$ implies that $\langle \bfq,\bfv\rangle=\langle \bfq+\bfp, W \bfv\rangle =0$.
Moreover, it is equivalent to
$$
\frac{1}{2}
\begin{pmatrix}
\dot p+\frac{\bar\bfu^\top\bar\bfp}{\|\bar\bfu\|}
\\
\frac{\dot p}{\|\bar\bfu\|} \bar\bfu
+
\frac{2\dot u}{{\dot u}-{\dot v}}
\bar \bfp
-\frac{{\dot u}+{\dot v}}{{\dot u}-{\dot v}}
\frac{\bar\bfu^\top\bar \bfp}{\|\bar\bfu\|^2}\bar\bfu
\end{pmatrix}
=
\frac{1}{2}
\begin{pmatrix}
\dot q -\frac{\bar\bfu^\top\bar \bfq}{\|\bar\bfu\|}
\\
-\frac{\dot q}{\|\bar\bfu\|} \bar\bfu
-\frac{2\dot v}{{\dot u}-{\dot v}} \bar \bfq
+\frac{{\dot u}+{\dot v}}{{\dot u}-{\dot v}}
\frac{\bar\bfu^\top\bar\bfq}{\|\bar\bfu\|^2}\bar\bfu
\end{pmatrix}.
$$
Note that $\dot v\neq 0$ and
${\dot v}\big(\dot q -\frac{\bar\bfu^\top\bar \bfq}{\|\bar\bfu\|}\big)
={\dot v} \langle \bfq, (1;-\frac{\bar\bfu}{\|\bar \bfu\|})\rangle
=\frac{\dot v}{\dot u}\langle \bfq, (\dot u;-\bar\bfu)\rangle
=\langle \bfq, \bfv\rangle
=0$.
Thus,
$$
\begin{array}{ll}
W\bfp=(I-W)\bfq
\\[1mm]
\Leftrightarrow
\dot p+\frac{\bar\bfu^\top\bar\bfp}{\|\bar\bfu\|}
=0
\ \mbox{and} \
-\frac{ \bar\bfu^\top\bar\bfp }{\|\bar\bfu\|^2} \bar\bfu
+
\frac{2\dot u}{{\dot u}-{\dot v}} \bar\bfp
-
\frac{{\dot u}+{\dot v}}{{\dot u}-{\dot v}}
\frac{\bar\bfu^\top\bar\bfp}{\|\bar\bfu\|^2}\bar\bfu
=
-\frac{\bar\bfu^\top\bar \bfq}{\|\bar\bfu\|^2} \bar\bfu
-\frac{2\dot v}{{\dot u}-{\dot v}} \bar \bfq
+\frac{{\dot u}+{\dot v}}{{\dot u}-{\dot v}}
\frac{\bar\bfu^\top\bar \bfq}{\|\bar\bfu\|^2}\bar\bfu
\\
\Leftrightarrow
\dot p+\frac{\bar\bfu^\top\bar\bfp}{\|\bar\bfu\|}
=0
\ \mbox{and} \
-\frac{ \bar\bfu \bar\bfu^\top}{\|\bar\bfu\|^2}
\bar\bfp
+
\frac{2\dot u}{{\dot u}-{\dot v}} \bar\bfp
-
\frac{{\dot u}+{\dot v}}{{\dot u}-{\dot v}}
\frac{\bar\bfu\bar\bfu^\top}{\|\bar\bfu\|^2}\bar\bfp
=
-\frac{\bar\bfu\bar\bfu^\top}{\|\bar\bfu\|^2}\bar \bfq
-\frac{2\dot v}{{\dot u}-{\dot v}} \bar \bfq
+\frac{{\dot u}+{\dot v}}{{\dot u}-{\dot v}}
\frac{\bar\bfu\bar\bfu^\top}{\|\bar\bfu\|^2}\bar \bfq
\\
\Leftrightarrow
\dot p+\frac{\bar\bfu^\top\bar\bfp}{\|\bar\bfu\|}
=0
\ \mbox{and} \
\frac{2\dot u}{{\dot u}-{\dot v}} \bar\bfp
-
\frac{2{\dot u}}{{\dot u}-{\dot v}}
\frac{\bar\bfu\bar\bfu^\top}{\|\bar\bfu\|^2}\bar\bfp
=
-\frac{2\dot v}{{\dot u}-{\dot v}} \bar \bfq
+\frac{2 {\dot v}}{{\dot u}-{\dot v}}
\frac{\bar\bfu\bar\bfu^\top}{\|\bar\bfu\|^2}\bar \bfq
\\
\Leftrightarrow
\dot p+\frac{\bar\bfu^\top\bar\bfp}{\|\bar\bfu\|}
=0
\ \mbox{and} \
(I_{r}-\frac{\bar\bfu\bar\bfu^\top}{\|\bar\bfu\|^2})\bar\bfp
=
-\frac{\dot v}{\dot u}
(I_r-\frac{\bar\bfu\bar\bfu^\top}{\|\bar\bfu\|^2})\bar \bfq
\\
\Leftrightarrow
\dot p+\frac{\bar\bfu^\top\bar\bfp}{\|\bar\bfu\|}
=0
\ \mbox{and} \
(I_{r}-\frac{\bar\bfu\bar\bfu^\top}{\|\bar\bfu\|^2})
(\bar\bfp+\frac{\dot v}{\dot u}\bar\bfq)=\bfzero
\\
\Leftrightarrow
\dot p+\frac{\bar\bfu^\top\bar\bfp}{\|\bar\bfu\|}
=0
\quad\mbox{and}\quad
\bar\bfp\in
\big\{\tau \frac{\bar \bfu}{\|\bar \bfu\|} -\frac{\dot v}{\dot u} \bar\bfq\mid \tau\in \Re\big\}
\\[1mm]
\Leftrightarrow
\bfp\in \left\{
\left(\frac{\dot v}{\dot u}\dot q-\tau;\,
-\frac{\dot v}{\dot u}\bar\bfq+\tau\frac{\bar\bfu}{\|\bar\bfu\|}\right)
\ \mid \
\tau\in\Re
\right\},
\end{array}
$$
which, together with \eqref{d3preliminary}, implies the explicit formulation of $\cD^*\cN_{\cQ}(\bfu,\bfv)(\bfq)$.

\smallskip
{\bf (4)}
Since $\bfu\in\boundary\cQ\setminus\{\bfzero\}$ and $\bfv=\bfzero$, one has
$ \bfu+\bfv=\bfu \in\boundary{\cQ}\setminus\{\bfzero\}$.
From Lemma \ref{codchange} we know that
$\bfp\in\cD^*\cN_{\cQ}(\bfu,\bfv)(\bfq)$ if and only if
$-\bfq\in\cD^{\ast}\Pi_{\cQ}(\bfu+\bfv)(-\bfp-\bfq)$,
which is equivalent, by Part $(4)$ of Lemma \ref{projcod}, to the following condition:
\[
\label{codequiv}
-\bfq\in
\begin{cases}
\conv\{-\bfp-\bfq, A(\bfu)(-\bfp-\bfq)\}, & \mbox{if } \langle (-\bfp-\bfq), (1;-\frac{\bar\bfu}{\|\bar\bfu\|}) \rangle \ge 0, \\
 \{-\bfp-\bfq, A(\bfu)(-\bfp-\bfq)\}, & \mbox{ otherwise},
\end{cases}
\quad
\forall \,\bfp,\bfq \in\Re^{1+r},
\]
where the definition of $A(\bfu)$ comes from \eqref{defab}.
Note that \eqref{codequiv} holds if and only if
\[
\label{pcondition}
\bfp=\frac{\varrho}{2}
\begin{pmatrix}
1
\\
-\frac{\bar{\bfu}}{\|\bar{\bfu}\|}
\end{pmatrix}
\begin{pmatrix}
1, -\frac{\bar{\bfu}^\top}{\|\bar{\bfu}\|}
\end{pmatrix}
{(\bfp+\bfq)},
\
\varrho\in
\begin{cases}
[0,1], & \mbox{if } -\langle \bfp , (1;-\frac{\bar\bfu}{\|\bar\bfu\|}) \rangle
\ge
\langle \bfq , (1;-\frac{\bar\bfu}{\|\bar\bfu\|}) \rangle,
\\
\{0,1\}, & \mbox{ otherwise}.
\end{cases}
\]
Moreover,  \eqref{pcondition} implies
\[
\label{check4}
\begin{array}{ll}
(1-\varrho)\langle (1; -\frac{\bar{\bfu}}{\|\bar{\bfu}\|}),\bfp\rangle
=\varrho\langle (1; -\frac{\bar{\bfu}}{\|\bar{\bfu}\|}),{\bfq}\rangle.
\end{array}
\]
For a given vector $\bfq\in\Re^{1+r}$, we consider the following three cases,
based on the value of
$\langle (1;-\frac{\bar\bfu}{\|\bar\bfu\|}),\bfq\rangle$, to get the equivalent conditions of
$\bfp\in\cD^*\cN_{\cQ}(\bfu,\bfv)(\bfq)$ from \eqref{pcondition}:

\begin{itemize}
\item[{\bf -}]
If $\langle (1;-\frac{\bar\bfu}{\|\bar\bfu\|}),\bfq\rangle>0$, one can not take $\varrho=1$,
or else \eqref{check4} fails to hold.
When $\varrho\in(0,1)$, by \eqref{check4} one should have
$\langle (1;-\frac{\bar\bfu}{\|\bar\bfu\|}),\bfp\rangle>0$,
but one can not find a vector $\bfp$ such that \eqref{pcondition} holds in this case.
Therefore, one can only take $\varrho=0$, so that $\bfp=\bfzero$ by the equality in \eqref{pcondition}.
In this case, it is easy to see that \eqref{pcondition} holds.
Thus, $\cD^*\cN_{\cQ}(\bfu,\bfv)(\bfq)=\{\bfzero\}$.

\item[{\bf -}]
If $\langle (1;-\frac{\bar\bfu}{\|\bar\bfu\|}),\bfq \rangle=0$,
by \eqref{check4} one has $\langle (1;-\frac{\bar\bfu}{\|\bar\bfu\|}),\bfp \rangle=0$
whenever $\varrho\in[0,1)$.
In this case, \eqref{pcondition} holds if and only if $\bfp=\bfzero$.
Moreover,  when $\varrho=1$, \eqref{check4} always holds.
In this case,  \eqref{pcondition} holds if and only if $\bfp=\frac{1}{2}(1,-\frac{\bar\bfu^\top}{\|\bar\bfu\|})\bfp (1;-\frac{\bar\bfu}{\|\bar\bfu\|} )$.
Thus,  $\cD^*\cN_{\cQ}(\bfu,\bfv)(\bfq)=\{\tau (1 ; -\frac{\bar \bfu}{\|\bar\bfu\|})\mid \tau \in\Re\}$.

\item[{\bf -}]
If $\langle (1;-\frac{\bar\bfu}{\|\bar\bfu\|}),\bfq \rangle<0$, one can not take $\varrho=1$, or else \eqref{check4} fails to hold.
When $\varrho=0$, it is easy to see that \eqref{pcondition} holds if and only if $\bfp=\bfzero$.
When $\varrho\in(0,1)$, by \eqref{check4} one should have
$\langle (1;-\frac{\bar\bfu}{\|\bar\bfu\|}),\bfp \rangle<0$.
Therefore, \eqref{pcondition} holds if and only if
$\bfp=\tau ( 1;-\frac{\bar\bfu}{\|\bar\bfu\|})$
with
$\tau:=\frac{\varrho}{2}(1,-\frac{\bar\bfu^\top}{\|\bar\bfu\|})(\bfp+\bfq)\le 0$.
Note that such a condition is valid for any $\tau<0$ since
$\varrho=\frac{2\tau}{2\tau+\langle(1;-\frac{\bar\bfu}{\|\bar\bfu\|}),\bfq \rangle}\in(0,1)$ always holds.
Thus, $\cD^*\cN_{\cQ}(\bfu,\bfv)(\bfq)=\{\tau (1 ; -\frac{\bar \bfu}{\|\bar\bfu\|})\mid \tau \le 0\}$.
\end{itemize}

{\bf (5)}
Since $\bfu=\bfzero$ and $\bfv\in-\boundary\cQ\setminus\{\bfzero\}$,
one has $ \bfu+\bfv=\bfv\in-\boundary\cQ\setminus\{\bfzero\}$.
From Lemma \ref{codchange} we know that
$\bfp\in\cD^*\cN_{\cQ}(\bfu,\bfv)(\bfq)$ if and only if
$-\bfq\in\cD^{\ast}\Pi_{\cQ}(\bfu+\bfv)(-\bfp-\bfq)$,
which is equivalent, by Part $(5)$ of Lemma \ref{projcod}, to the condition
\[
\label{codequiv5}
-\bfq\in
\begin{cases}
\conv\{\bfzero, B(\bfv)(-\bfp-\bfq)\}, & \mbox{if } \langle -\bfp-\bfq, (1;\frac{\bar\bfv}{\|\bar\bfv\|}) \rangle \ge 0, \\
 \{\bfzero, B(\bfv)(-\bfp-\bfq)\}, & \mbox{ otherwise},
\end{cases}
\quad
\forall \, \bfp,\bfq \in\Re^{1+r},
\]
where the definition of $B(\bfv)$ comes from \eqref{defab}.
Note that \eqref{codequiv5} holds if and only if
\[
\label{pcondition5}
\bfq=\frac{\varrho}{2}
\begin{pmatrix}
1
\\
\frac{\bar{\bfv}}{\|\bar{\bfv}\|}
\end{pmatrix}
\begin{pmatrix}
1, \frac{\bar{\bfv}^\top}{\|\bar{\bfv}\|}
\end{pmatrix}
{(\bfp+\bfq)}
\quad
\mbox{with}
\quad
\varrho\in
\begin{cases}
[0,1], & \mbox{if } -\langle \bfp , (1;\frac{\bar\bfv}{\|\bar\bfv\|}) \rangle
\ge
\langle \bfq , (1;\frac{\bar\bfv}{\|\bar\bfv\|}) \rangle,
\\
\{0,1\}, & \mbox{ otherwise}.
\end{cases}
\]
Therefore, one should have
$\bfq=\tau (1 ; \frac{\bar \bfv}{\|\bar\bfv\|})$ for some $\tau\in\Re$,
and
\[
\label{check5}
\begin{array}{ll}
(1-\varrho)
\langle (1; \frac{\bar{\bfv}}{\|\bar{\bfv}\|}
),\bfq\rangle
=\varrho
\langle (1; \frac{\bar{\bfv}}{\|\bar{\bfv}\|}),
{\bfp}\rangle.
\end{array}
\]
Consequently, one can equivalently reformulate \eqref{pcondition5} to the condition that
\[
\label{pcondition52}
(1-\varrho)\tau
\begin{pmatrix}
1
\\
\frac{\bar{\bfv}}{\|\bar{\bfv}\|}
\end{pmatrix}
=
\frac{\varrho}{2}
\begin{pmatrix}
1
\\
\frac{\bar{\bfv}}{\|\bar{\bfv}\|}
\end{pmatrix}
\begin{pmatrix}
1, \frac{\bar{\bfv}^\top}{\|\bar{\bfv}\|}
\end{pmatrix}
{\bfp}
\quad
\mbox{with}
\quad
\varrho\in
\begin{cases}
[0,1], & \mbox{if } -\langle \bfp , (1;\frac{\bar\bfv}{\|\bar\bfv\|}) \rangle
\ge
2\tau,
\\
\{0,1\}, & \mbox{ otherwise}.
\end{cases}
\]
\begin{itemize}
\item[{\bf -}]
If  $\tau=0$, one can take $\varrho=0$ such that \eqref{pcondition52} holds for any $\bfp\in\Re^{1+r}$.
Thus, $\cD^*\cN_{\cQ}(\bfu,\bfv)(\bfzero)=\Re^{1+r}$.

\item[{\bf -}]
If $\tau>0$,  one can not take $\varrho=0$, or else \eqref{check5} fails to hold.
When $\varrho\in(0,1)$, by \eqref{check5} one should have
$\langle (1;\frac{\bar\bfv}{\|\bar\bfv\|}), \bfp \rangle >0$,
but one can not find a vector $\bfp\in\Re^{1+r}$ such that
$-\langle \bfp , (1;\frac{\bar\bfv}{\|\bar\bfv\|}) \rangle\ge 2\tau$.
Therefore, one can only take $\varrho=1$, so that $\langle (1; \frac{\bar{\bfv}}{\|\bar{\bfv}\|}),
{\bfp}\rangle=0$ by \eqref{check5}.
In this case, it is easy to see that \eqref{pcondition52} holds.
Thus,
$\cD^*\cN_{\cQ}(\bfu,\bfv)(\bfq )=\{\bfp\mid \langle \bfp,\bfq \rangle=0\}$.

\item[{\bf -}]
If $\tau <0$, one has $\varrho\neq 0$ by the equality in \eqref{pcondition52}.
When $\varrho=1$,  \eqref{pcondition52} holds if and only if  $\langle (1;\frac{\bar\bfv}{\|\bar\bfv\|}), \bfp\rangle=0$.
When $\varrho\in(0,1)$, one has
$
-\langle (1;\frac{\bar\bfv}{\|\bar\bfv\|}), \bfp\rangle
=-\frac{1-\varrho}{\varrho}\langle (1;\frac{\bar\bfv}{\|\bar\bfv\|}),\bfq\rangle
=\frac{-2\tau(1-\varrho)}{\varrho}>2\tau$ if \eqref{check5} holds.
Thus, \eqref{pcondition52} holds if and only if there exists $\varrho\in(0,1)$ such that $\frac{\varrho}{2(1-\varrho)\tau}
\begin{pmatrix}
1, \frac{\bar{\bfv}^\top}{\|\bar{\bfv}\|}
\end{pmatrix}
{\bfp}=1$. Since $\tau<0$, we know that \eqref{pcondition52} holds if and only if
$\langle (1;\frac{\bar\bfv}{\|\bar\bfv\|}),\bfp\rangle < 0$.
Consequently, one has in this case that
$\cD^*\cN_{\cQ}(\bfu,\bfv)(\bfq)=\{\bfp\mid \langle \bfp, \bfq\rangle\ge 0\}$.
\end{itemize}

{\bf (6a)}
Since $\bfu=\bfv=\bfzero$,
by \eqref{partialbzero} and Part $(6)$ of Lemma \ref{projcod} one has
$\bfzero\in {\partial}_{B}\Pi_{\cQ}(\bfu+\bfv)\bfr\subseteq\cD^*\Pi_{\cQ}(\bfu+\bfv)(\bfr)$ for all $\bfr\in\Re^{1+r}$.
Then by Lemma \ref{codchange} one has
$\cD^*\cN_{\cQ}(\bfu, \bfv)(\bfzero)=\Re^{1+r}$.

\smallskip
{\bf (6b)}
Since $\bfu=\bfv=\bfzero$,
by Part (6) of Lemma \ref{projcod} one has
$
\big\{\cQ\cap (\bfr-\cQ)\big\}
\subseteq
\cD^{\ast}\Pi_{\cQ}(\bfu+\bfv)(\bfr)$ for all $\bfr\in\Re^{1+r}$.
Therefore, for any $\bfq\in-\cQ$ and $\bfp\in-\cQ$, by letting $\bfr=-\bfp-\bfq$
one has $-\bfq =\bfr-(-\bfp)\in (\bfr-\cQ)$.
Consequently,
$\bfp\in\cD^*\cN_{\cQ}(\bfu,\bfv)(\bfq)$
for any  $\bfq\in-\cQ$ and $\bfp\in-\cQ$.

\smallskip
{\bf (6c)}
Since $\bfu=\bfv=\bfzero$,
by Part (6) of Lemma \ref{projcod} one has for any $\bfr\in\Re^{1+r}$, it holds
$\bfs\in \cD^{\ast}\Pi_{\cQ}(\bfu+\bfv)(\bfr)$, or equivalently (by Lemma \ref{codchange}),
\[
\label{coderiv6}
\bfs-\bfr \in\cD^*\cN_{\cQ}(\bfu,\bfv)(-\bfs),
\quad \forall \, \bfr\in\Re^{1+r},
\]
if $s\in\conv\{\bfr, A\bfr\}$,
where $A=
I_{1+r}-\frac{1}{2}  \begin{pmatrix}  1 \\ -\bfw \end{pmatrix}(1,-\bfw^\top)$
for some $\bfw \in\Re^{r}$ such that $\|\bfw\|=1$ and $\langle\bfr, (1; -\bfw) \rangle \geq 0$.
Since $\bfp\in -\boundary\cQ\setminus\{\bfzero\}$ and $\bfq\in\Re^{1+r}$ with $\langle \bfp,\bfq\rangle \ge 0$,
by letting $\bfs:=-\bfq$ and $\bfr:=-\bfp-\bfq$, one can see that $\bfp=\bfs-\bfr$.
Meanwhile, by letting  $\bfw:=\frac{\bar\bfp}{\|\bar\bfp\|}$,
one can get that
\[
\label{coderi6}
\begin{array}{ll}
\langle \bfr,(1;-\bfw) \rangle
=-\langle\bfp+\bfq, (1;-\bfw) \rangle
=-\dot p +\langle\bar\bfp,\bfw\rangle-\dot q +\langle \bar\bfq,\bfw \rangle
\\
=-\dot p+\langle\bar\bfp,\frac{\bar\bfp}{\|\bar\bfp\|}\rangle
-\dot q +\langle \bar\bfq,\frac{\bar\bfp}{\|\bar\bfp\|} \rangle
=-\dot p+\|\bar\bfp\|
+\frac{1}{\|\bar\bfp\|}(\dot q\dot p +\langle \bar\bfq,\bar\bfp \rangle)
=2\|\bar\bfp\| +\frac{\langle \bfp,\bfq  \rangle}{\|\bar\bfp\|} >0.
\end{array}
\]
In this case, one can further  define  $\varrho:=\frac{\|\bfp\|^2}{\langle\bfp,\bfq\rangle+\|\bfp\|^2}\in(0,1]$ and obtain that
$$
\begin{array}{ll}
(1-\varrho)\bfr+\varrho A\bfr=
(1-\varrho)(-\bfp-\bfq)
+\varrho \left[ I_{1+r}-\frac{1}{2}  \begin{pmatrix}  1 \\ -\bfw \end{pmatrix}(1,-\bfw^\top)  \right](-\bfp-\bfq)
\\[2mm]
= -\bfq-\bfp
+\frac{\varrho}{2} \begin{pmatrix}  1 \\ -\bfw \end{pmatrix}(1,-\bfw^\top) \bfq
+\frac{\varrho}{2} \begin{pmatrix}  1 \\ -\bfw \end{pmatrix}(1,-\bfw^\top) \bfp
\\[2mm]
=-\bfq-\bfp
+\frac{\varrho}{2} (\dot q -\frac{\bar\bfp^\top\bar\bfq}{\|\bar\bfp\|})\begin{pmatrix}  1 \\ -\bfw \end{pmatrix}
+\frac{\varrho}{2} (\dot p -\frac{\bar\bfp^\top \bar\bfp}{\|\bar\bfp\|})
\begin{pmatrix}  1 \\ -\bfw \end{pmatrix}
\\[2mm]
=-\bfq-\bfp
-\varrho\left(\frac{\langle\bfp,\bfq\rangle}{2\|\bar\bfp\|}+\|\bar\bfp\|\right)
\begin{pmatrix}  1 \\-\frac{\bar\bfp}{\|\bar\bfp\|} \end{pmatrix}
=-\bfq-\bfp+\varrho\frac{\langle\bfp,\bfq\rangle+\|\bfp\|^2}{\|\bfp\|^2}\begin{pmatrix}\dot p \\ \bar\bfp\end{pmatrix}=-\bfq,
\end{array}
$$
which, together with \eqref{coderi6} and the fact that $\varrho\in(0,1]$,
implies $\bfp\in\cD^*\cN_{\cQ}(\bfu,\bfv)(\bfq)$
from \eqref{coderiv6}.

\smallskip
{\bf (6d)}
Since $\bfu=\bfv=\bfzero$,
by Part (6) of Lemma \ref{projcod} one has for any $\bfr\in\Re^{1+r}$, it holds that
$\bfs\in \cD^{\ast}\Pi_{\cQ}(\bfu+\bfv)(\bfr)$, or equivalently (by Lemma \ref{codchange}),
\[
\label{coder6666}
\bfs-\bfr \in\cD^*\cN_{\cQ}(\bfu,\bfv)(-\bfs),
\quad \forall \, \bfr\in\Re^{1+r},
\]
if $s\in\conv\{\bfzero, B\bfr\}$,
where $B=
\frac{1}{2}  \begin{pmatrix}  1 \\ \bfw \end{pmatrix}(1,\bfw^\top)
$ for some $\bfw \in\Re^{r}$ such that $\|\bfw\|=1$ and $\langle \bfr, (1; \bfw)\rangle \geq 0$.
Let $\bfp\in\Re^{1+r}$ and $\bfq\in -\boundary\cQ\setminus\{\bfzero\}$ such that $\langle \bfp,\bfq\rangle \ge 0$.
By letting $\bfs:=-\bfq$ and $\bfr:=-\bfp-\bfq$, one can see that $\bfp=\bfs-\bfr$.
Meanwhile, by letting  $\bfw:=-\frac{\bar\bfq}{\|\bar\bfq\|}$,
one can see that
\[
\label{coder666}
\begin{array}{ll}
\langle \bfr,(1;\bfw) \rangle
=-\langle\bfp+\bfq, (1;\bfw) \rangle
=-\dot p -\langle\bar\bfp,\bfw\rangle-\dot q -\langle \bar\bfq,\bfw \rangle
\\
=-\dot p -\langle\bar\bfp,-\frac{\bar\bfq}{\|\bar\bfq\|}\rangle-\dot q -\langle \bar\bfq,-\frac{\bar\bfq}{\|\bar\bfq\|} \rangle
=-\dot p +\langle\bar\bfp,\frac{\bar\bfq}{\|\bar\bfq\|}\rangle-\dot q +\|\bar\bfq\|
=\frac{\langle \bfq, \bfp\rangle}{\|\bar\bfq\|} +2\|\bar\bfq\|>0.
\end{array}
\]
In this case, one can further  define  $\varrho:=\frac{\|\bfq\|^2}{\langle\bfp,\bfq\rangle+\|\bfq\|^2}\in(0,1]$ and obtain that
$$
\begin{array}{ll}
\varrho B\bfr=
\frac{\varrho}{2} \begin{pmatrix}  1 \\ \bfw \end{pmatrix}(1,\bfw^\top) (-\bfp-\bfq)
= -\frac{\varrho}{2}  \begin{pmatrix}  1 \\ \bfw \end{pmatrix}(1,\bfw^\top) \bfq
-\frac{\varrho}{2}  \begin{pmatrix}  1 \\ \bfw \end{pmatrix}(1,\bfw^\top) \bfp
\\[4mm]
=-\frac{\varrho}{2}(\dot p-\frac{\bar\bfq^\top\bar\bfp}{\|\bar\bfq\|})
\begin{pmatrix}1\\ \bfw  \end{pmatrix}
-\frac{\varrho}{2} (\dot q -\frac{\bar\bfq^\top\bar\bfq}{\|\bar\bfq\|})
\begin{pmatrix} 1\\ \bfw\end{pmatrix}
=\varrho \frac{\langle \bfp,\bfq \rangle+\|\bfq\|^2}{2\|\bar\bfq\|}
\begin{pmatrix} 1\\-\frac{\bar\bfq}{\|\bar\bfq\|}\end{pmatrix}
=-\bfq.
\end{array}
$$
This, together with \eqref{coder6666}, \eqref{coder666}, and the fact that $\varrho\in(0,1]$,
implies  that
$\bfp\in\cD^*\cN_{\cQ}(\bfu,\bfv)(\bfq)$ holds.

\smallskip
{\bf (6e)}
Since $\bfu=\bfv=\bfzero$,
by Part (6) of Lemma \ref{projcod} and \eqref{partialbzero} one has for any $\bfr\in\Re^{1+r}$, it holds $\bfr\in \cD^{\ast}\Pi_{\cQ}(\bfu+\bfv)(\bfr)$.
This is equivalent, by Lemma \ref{codchange}, to
$\bfzero \in\cD^*\cN_{\cQ}(\bfu,\bfv)(-\bfr)$ for all $\bfr\in\Re^{1+r}$.
This completes the proof of the lemma.
\end{proof}

\section{Equivalence of the Aubin property and the strong regularity}
\label{sec:main}
This section establishes the main result of this paper.
The main step is to show that the strong second-order sufficient condition in Definition \ref{defsoccs} holds.  
We will reformulate the Mordukhovich criterion for the Aubin property of $\rS_{\KKT}$ to the necessary condition in Corollary \ref{aubinss}. 
Moreover, by specifying a basis for the affine hull of the critical cone 
in Corollary \ref{aubins2}, we can put our discussion in a space with a lower dimension (Lemma \ref{lemsosc}). 
Finally, by extensively utilizing Corollary \ref{aubinss} and Lemma \ref{lemtech}, we obtain the strong second-order sufficient condition by Propositions \ref{propreduce} and \ref{propmain}, in which successive dimension reduction is involved. 

Let  $\bfx^*$ be a locally optimal solution to \eqref{nlsocp} with
$\bfy^*=(\bflambda^*,\bfmu^*)$ being the associated multiplier.
Note that any index $j\in \{1,\ldots, J\}$ belongs to one and only one of the following $6$ sets:
\[
\label{jdecompose}
\left\{\begin{array}{ll}
J_1:=\left\{ j \mid g^j(\bfx^*) =\bfzero, (\bfmu^*)^j \in \interior\cQ_j \right\},
\\[1mm]
J_2:=\left\{ j \mid g^j(\bfx^*)\in \interior\cQ_j, (\bfmu^*)^j =\bfzero \right\},
\\[1mm]
J_3:=\left\{ j \mid g^j(\bfx^*), (\bfmu^*)^j \in \boundary{\cQ_j}\setminus\{\bfzero\} \right\},
\\[1mm]
J_4:=\left\{ j \mid g^j(\bfx^*)=\bfzero, (\bfmu^*)^j \in \boundary{\cQ_j}\setminus\{\bfzero\} \right\},
\\[1mm]
J_5:=\left\{ j \mid  g^j(\bfx^*)\in \boundary{\cQ_j}\setminus\{\bfzero\}, (\bfmu^*)^j=\bfzero \right\},
\\[1mm]
J_6:=\left\{ j \mid g^j(\bfx^*) =\bfzero, (\bfmu^*)^j =\bfzero \right\}.
\end{array}\right.
\]
For simplicity, $J_1 \cup J_2$ is abbreviated to $J_{1,2}$,   $J_1 \cup J_2 \cup J_3$ is abbreviated to $J_{1,2,3}$, etc. For an index set $I\subseteq \{1,2,3,4,5,6\}$, $|J_I|$ denotes the cardinality of the set $J_I$.
Note that when $r_j=0$, $j$ must belong to $J_{1,2,6}$.
Recall from \cite[Lemma 25]{bonnans2005}  that, for each $j=1,\ldots, J$, the tangent cone of $\cQ_j$ at $g^j(\bfx^*)$ is given by
\[
\label{tancone}
\cT_{\cQ_j}(g^j(\bfx^*))=
\begin{cases}
\Re^{1+r_j},  & \mbox{ if } j\in J_2,
\\
\cQ_j,  & \mbox{ if } j\in J_{1,4,6},
\\
\{\bfv\in\Re^{1+r_j} \mid  \langle \bar\bfv,\bar g^j(\bfx^*)\rangle-\dot{v}\dot g^j(\bfx^*) \le  0\},
& \mbox{ if } j\in J_{3,5}.
\end{cases}
\]
Furthermore, by \cite[Corollary 26]{bonnans2005} the critical cone $\cC(\bfx^*)$ can be given by
\[
\label{criticalcone2}
\cC(\bfx^*)
=\left\{
\bfd\in\Re^n\middle|
\begin{array}{ll}
\cJ h(\bfx^*)\bfd=\bfzero,
\\
\cJ g^j(\bfx^*)\bfd=\bfzero, &
\forall\, j\in J_1,
\\
\cJ g^j(\bfx^*)\bfd\in\cT_{\cQ_j}(g^j(\bfx^*)), &
\forall\, j \in  J_{5,6},
\\
\cJ g^j(\bfx^*)\bfd \in\Re_+ (({\dot\mu^*})^j;-({\bar\bfmu^*})^j),
&\forall\, j\in J_4,
\\
\langle \cJ g^j(\bfx^*)\bfd, (\bfmu^*)^j\rangle =0,
&\forall\,j\in J_3
\end{array}
\right\}.
\]
Moreover, according to \cite[eq. (47)]{bonnans2005}, the affine hull of the critical cone $\cC(\bfx^*)$, i.e. the smallest subspace that contains the critical cone, can be formulated as
\[
\label{affcricdef2}
\aff(\cC(\bfx^*))
=\left\{
\bfd\in\Re^n\middle|
\begin{array}{ll}
\cJ h(\bfx^*)\bfd=\bfzero,
\\
\cJ g^j(\bfx^*)\bfd=\bfzero, &\forall\,
j\in J_1,
\\
\cJ g^j(\bfx^*)\bfd\in\Re (({\dot\mu^*})^j;-({\bar\bfmu^*})^j), &\forall\, j\in J_4,
\\
\langle \cJ g^j(\bfx^*)\bfd, (\bfmu^*)^j\rangle =0,  &\forall\, j\in J_3
\end{array}
\right\}.
\]
In the following definition,  the second-order and strong second-order sufficient conditions for   \eqref{nlsocp} come from \cite[Theorems 29 \& 30]{bonnans2005}.
\begin{definition}
\label{defsoccs}
Let $\bfx^*$ be a stationary point of the nonlinear SOCP \eqref{nlsocp}.
We say that the second-order sufficient condition holds at $\bfx^*$ if there exists an associated multiplier $(\bflambda^*,\bfmu^*)$ such that
\[
\label{sosc}
\langle \bfd,\left(\nabla^2_{\bfx\bfx}\cL (\bfx^*,\bflambda^*,\bfmu^*)
+\cH(\bfx^*,\bfmu^*)\right)\bfd\rangle
>0,
\quad
\forall\, \bfd\in \cC(\bfx^*)\setminus\{\bfzero\},
\]
where $\nabla_{\bfx\bfx}^2\cL$ is the Hessian of $\cL$ with respect to $\bfx$,
and $\cH(\bfx^*,\bfmu^*):=\sum_{j=1}^J \cH_j(\bfx^*,\bfmu^*)$ with
$$
\cH_j(\bfx^*,\bfmu^*):=
\left\{
\begin{array}{ll}
\ds-\frac{(\dot\mu^*)^j}{\dot{g}^j(\bfx^*)}\nabla g^j(\bfx^*)
\begin{pmatrix} 1& \bfzero^\top\\ \bfzero& -I_{r_j} \end{pmatrix}
\cJ g^j(\bfx^*),
&\mbox{ if } g^j(\bfx^*)\in\boundary\cQ_j\setminus\{\bfzero\},
\\[3mm]
O_{n},&\mbox{ otherwise}.
\end{array}
\right.
$$
We say that the strong second-order sufficient condition holds at $\bfx^*$ if there exists an associated multiplier $(\bflambda^*,\bfmu^*)$ such that
\[
\label{ssosc}
\left< \bfd,\left(\nabla^2_{\bfx\bfx}\cL(\bfx^*,\bflambda^*,\bfmu^*)
+\cH(\bfx^*,\bfmu^*)\right)\bfd\right>
>0,
\quad
\forall\, \bfd\in \aff(\cC(\bfx^*))\setminus\{\bfzero\}.
\]
Similarly, we say that the strong second-order necessary condition holds at $\bfx^*$ if there exists an associated multiplier $(\bflambda^*,\bfmu^*)$ such that
\[
\label{ssonc}
\left< \bfd,\left(\nabla^2_{\bfx\bfx}\cL(\bfx^*,\bflambda^*,\bfmu^*)
+\cH(\bfx^*,\bfmu^*)\right)\bfd\right>
\ge 0,
\quad
\forall\, \bfd\in \aff(\cC(\bfx^*)).
\]
\end{definition} 
\begin{remark}
\label{rmkvs}
We will show that 
$\left(\nabla^2_{\bfx\bfx}\cL(\bfx^*,\bflambda^*,\bfmu^*)
+\cH(\bfx^*,\bfmu^*)\right)\bfd\neq\bfzero$ for all $\bfd\in \aff(\cC(\bfx^*))\setminus\{\bfzero\}$
(Corollary \ref{aubinss}) if the Aubin property of $\rS_{\KKT}$ holds. 
Consequently, we devote our main effort to proving that the strong second-order necessary condition \eqref{ssonc} holds to obtain Proposition \ref{propmain}. 
In \cite[Theorem 4.2]{benko24}, by assuming that the variational sufficiency for local optimality \cite{rock2023} holds, the authors established the equivalence between the Aubin property of $\rS_{\KKT}$ and the strong regularity of \eqref{ge}   for stationary points.  
It can be observed from \cite[Theorem 3 \& Example 1]{rock2023} that for the conventional nonlinear programming, i.e., a special case of \eqref{op} with $r_1=\cdots=r_J=0$, the variational sufficiency for local optimality is a sufficient condition for the strong second-order necessary condition \eqref{ssonc}. 
Therefore, making such an assumption is very close to directly assuming the strong second-order sufficient condition \eqref{ssosc}. 
In contrast, we will derive the strong second-order sufficient condition solely from the Aubin property of $\rS_{\KKT}$ without assuming the variational sufficiency as in  \cite{rock2023}. 
\end{remark}

\subsection{Implications of the Aubin property}
\label{sec:aubinimply}
In this subsection,
we introduce a reduction approach to the
Aubin property of $\rS_{\rm GE}$ \eqref{sge}
characterized by the Mordukhovich criterion,
for recasting it into a more accessible form.
For convenience, we make the following assumption.
\begin{assumption}
\label{assblanket}
The point $\bfx^*$ is a nondegenerate locally optimal solution
{\rm(}i.e., \eqref{cnd} holds at $\bfx^*${\rm)}
to   \eqref{nlsocp} with $\bfy^*=(\bflambda^*,\bfmu^*)$ being the associated multiplier,
and the index set $\{1,\ldots,J\}$ are classified by \eqref{jdecompose}.
\end{assumption}

Thanks to Lemma \ref{lemma:aubin}, we can get the following coderivative-based characterization for the Aubin property of ${\rS}_{\rm GE}$ \eqref{sge}.
\begin{lemma}
\label{aubproof2}
Under Assumption \ref{assblanket}, the mapping ${\rS}_{\rm GE}$ in \eqref{sge} has the Aubin property at the origin for $\bfx^*$
if and only if for all $\bfd\in\aff(\cC(\bfx^*))\backslash\{\bfzero\}$ it holds that
\[
\label{aubinreduce1}
\begin{array}{lr}
\ds
\nabla_{\bfx\bfx}^2\cL(\bfx^*,\bflambda^*,\bfmu^*)\bfd
\notin \range \nabla h(\bfx^*)
\\[1mm]
\qquad\qquad
\ds
-\left\{
\sum_{j=1}^J\nabla g^j(\bfx^*)\bfp^j
\ \Big\vert\
\bfp^j\in D^*\cN_{\cQ_j}(g^j(\bfx^*),-(\bfmu^*)^j)(\cJ g^j(\bfx^*)\bfd)
\right\}.
\end{array}
\]
\end{lemma}

\begin{proof}
Since Assumption \ref{assblanket} holds, according to
Lemmas \ref{lemma:aubinorigion}
and \ref{lemma:aubin} we know that the
mapping $\rS_{\rm GE}$ in \eqref{sge} has the Aubin property at the origin for $\bfx^*$
if and only if for all $\bfd\, \in\Re^n\setminus\{\bfzero\}$,
\[
\label{aubinoriginal}
\begin{array}{lll}
\nabla_{\bfx\bfx}^2\cL(\bfx^*,\bflambda^*,\bfmu^*)\bfd
&
\notin
&-
\nabla h(\bfx^*)\cD^*\cN_{\{\bfzero\}}(h(\bfx^*),-\bflambda^*)(\cJ h(\bfx^*)\bfd)
\\
&&-\sum_{j=1}^J\nabla g^j(\bfx^*)\cD^*\cN_{\cQ_j}(g^j(\bfx^*),-(\bfmu^*)^j)(\cJ g^j(\bfx^*)\bfd).
\end{array}
\]
From \eqref{coderzero} we know that
$
D^*\cN_{\{\bfzero\}}(h(\bfx^*),-\bflambda^*)(\cJ h(\bfx^*)\bfd)
=\Re^m$, if $\cJ h(\bfx^*)\bfd=\bfzero$, or else it is an empty set.
Therefore, \eqref{aubinoriginal} holds whenever $\cJ h(\bfx^*)\bfd\neq \bfzero$.
Next, if $\bfzero\neq \bfd \notin\aff(\cC(\bfx^*))$ satisfies $\cJ h(\bfx^*)\bfd=\bfzero$,
according to \eqref{affcricdef2}, one can always choose an index $j$ such that at least one of the following three conditions holds:
$$
\begin{cases}
j\in J_1 \quad\mbox{and}\quad \cJ g^j(\bfx^*)\bfd\neq\bfzero,
\\
j\in J_4 \quad\mbox{and}\quad
\cJ g^j(\bfx^*)\bfd\notin \Re (({\dot\mu^*})^j;-({\bar\bfmu^*})^j),
\\
j\in J_3 \quad\mbox{and}\quad
\langle \cJ g^j(\bfx^*)\bfd, (\bfmu^*)^j\rangle \neq 0.
\end{cases}
$$
\begin{description}
\item[-]
If $j\in J_1$  and $\cJ g^j(\bfx^*)\bfd\neq\bfzero$,
by \eqref{jdecompose} one has $g^j(\bfx^*)=\bfzero$ and $(\bfmu^*)^j\in\interior\cQ_j$.
Then one can get
$\cD^*\cN_{\cQ_j}(g^j(\bfx^*),-(\bfmu^*)^j)(\cJ g^j(\bfx^*)\bfd)=\emptyset$
by Part (1) of Lemma \ref{codersecond}
(or \eqref{coderzerofirst} if $r_j=0$).

\item[-]
If $j\in J_4$  and
$\cJ g^j(\bfx^*)\bfd\notin \Re (({\dot\mu^*})^j;-(\bar{\bfmu}^*)^j )$ holds,
one can get from \eqref{jdecompose} that
$g^j(\bfx^*)=\bfzero$ and
$(\bfmu^*)^j\in\boundary\cQ_j\setminus\{\bfzero\}$,
so that $(\dot\mu^*)^j=\|(\bar\bfmu^*)^j\|$.
Then by Part $(5)$ of Lemma \ref{codersecond} one has
$\cD^*\cN_{\cQ_j}(g^j(\bfx^*),-(\bfmu^*)^j)(\cJ g^j(\bfx^*)\bfd)=\emptyset$.

\item[-]
If $j\in J_3$ and $\langle \cJ g^j(\bfx^*)\bfd, (\bfmu^*)^j\rangle \neq 0$,
one has from \eqref{jdecompose} that $(\bfmu^*)^j,g^j(\bfx^*) \in\boundary\cQ_j\setminus\{\bfzero\}$.
Then one has
$\cD^*\cN_{\cQ_j}(g^j(\bfx^*),-(\bfmu^*)^j)(\cJ g^j(\bfx^*)\bfd)=\emptyset$ by Part (3) of Lemma \ref{codersecond}.
\end{description}
In summary, according to the above discussions, we know that
\eqref{aubinoriginal} holds automatically whenever
$\bfd\notin\aff(\cC(\bfx^*))\backslash\{\bfzero\}$.
Therefore, \eqref{aubinreduce1} is equivalent to \eqref{aubinoriginal}.
\end{proof}

The following lemma, as an application of the reduction approach developed in \cite[Section 3.4.4]{B&S2000} for constraint systems, provides a convenient tool for characterizing the affine hull of the critical cone \eqref{affcricdef2} via the null space of a given matrix.

\begin{lemma}
\label{lem:aff}
Suppose that Assumption \ref{assblanket} holds.
Define the matrix
\[
\label{defXi}
\Xi:=\left(
\cJ h(\bfx^*);\,
\{\cJ g^{j}(\bfx^*)\}_{j\in J_1}; \,
\{[(\bfmu^*)^j]^\top\cJ g^j (\bfx^*) \}_{j\in J_3}; \,
\{S_j^\top\cJ g^j(\bfx^*)\}_{j\in J_4}
\right),
\]
where the columns of each matrix $S_j\in\Re^{(1+r_j)\times r_j}$ span a basis of $\{((\dot\mu^*)^j;-(\bar\bfmu^*)^j)\}^\perp$.
Then, one has $\aff(\cC(\bfx^*))=\ker \Xi$.
Moreover, $\Xi$ has full row rank.
\end{lemma}

\begin{proof}
The first conclusion comes immediately from \eqref{affcricdef2} and  \eqref{defXi}.
Next, we show that $\Xi$ is surjective.
Let $\bfzeta^0\in \Re^m$,
$\{\bfzeta^j\in \Re^{1+r_j}\}_{j\in J_1}$,
$\{\zeta^j\in \Re\}_{j\in J_3}$,
and
$\{\hat\bfzeta^j\in \Re^{r_j}\}_{j\in J_4}$
be arbitrarily given.
Define
$$
\left\{
\bfzeta^j:=\bfzero\in \Re^{1+r_j}
\right\}_{j\in J_{2,5,6}},
\left\{\bfzeta^j:=\frac{\zeta^j(\bfmu^*)^j}{\|(\bfmu^*)^j\|^2}\right\}_{j\in J_3},
\mbox{and}
\ \left\{
\bfzeta^j:=S_j(S_j^\top S_j)^{-1}\hat\bfzeta^j
\right\}_{j\in J_4}.
$$
One can see from \eqref{cnd} that there exist vectors $\bfd\in\Re^n$ and $\bfv^j\in\lin\left(\cT_{\cQ_j}\big(g^j(\bfx^*)\big)\right)$
such that
$\cJ h(\bfx^*)\bfd=\bfzeta^0$
{and}
$\cJ g^j(\bfx^*)\bfd+ \bfv^j = \bfzeta^j$ for all  $j=1,\ldots, J$, where by \eqref{tancone} the linearity space of the tangent cone takes the following form
\[
\label{lintan}
\lin(\cT_{\cQ_j}(g^j(\bfx^*)))
=
\begin{cases}
\Re^{1+r_j},  & \mbox{ if } j\in J_2,
\\
\{\bfzero\},  & \mbox{ if } j\in J_{1,4,6},
\\
\{\bfv\in\Re^{1+r_j} \mid  \langle \bar\bfv,\bar g^j(\bfx^*)\rangle-\dot{v}\dot g^j(\bfx^*) = 0\},
&\mbox{ if } j\in J_{3,5}.
\end{cases}
\]
Then, by using \eqref{lintan} one has
$\cJ h(\bfx^*)\bfd=\bfzeta^0$,
$\{\cJ g^{j}(\bfx^*)\bfd=\bfzeta^j\}_{j\in J_1}$,
$\{\cJ g^{j}(\bfx^*)\bfd+\bfv^{j}=\bfzeta^j\}_{j\in J_3}$
and $\{\cJ g^{j}(\bfx^*)\bfd=\bfzeta^j\}_{j\in J_4}$.
Moreover, for all $j\in J_3$, it holds that
$$
\langle \cJ g^j (\bfx^*)\bfd, (\bfmu^*)^j\rangle=\langle \bfzeta^j-\bfv^{j}, (\bfmu^*)^j\rangle=\langle \bfzeta^j, (\bfmu^*)^j\rangle
=\zeta^j.
$$
Meanwhile, for all $j\in J_4$, it holds that $S_j^\top\cJ g^j(\bfx^*)\bfd=S_j^\top\bfzeta^j=\hat \bfzeta^j$.
Therefore, we know that $\Xi$ has full row rank since it is {surjective}. This completes the proof.
\end{proof}

Under Assumption \ref{assblanket},
one can define the symmetric matrix  $\Theta\in\Re^{n\times n}$ by
\[
\label{ssocq}
\Theta:=
 \nabla_{\bfx\bfx}^2\cL(\bfx^*,\bflambda^*,\bfmu^*)
-\sum_{j\in J_3}\frac{(\dot\mu^*)^j}{\dot g^j(\bfx^*)}
\nabla g^j(\bfx^*)
\begin{pmatrix}
1 & \bfzero^\top\\
\bfzero & -I_{r_j}
\end{pmatrix}
\cJ g^{j}(\bfx^*).
\]
Based on Lemma \ref{lem:aff}, we further reduce the condition in Lemma \ref{aubproof2} for characterizing the Aubin property of the solution mapping  ${\rS}_{\rm GE}$ \eqref{sge}.

\begin{proposition}
\label{aubins}
Suppose that Assumption \ref{assblanket} holds.
Let $\Theta\in\Re^{n\times n}$  be the matrix defined in \eqref{ssocq}, and $H\in\Re^{n\times \ell}$ be a matrix with full column rank such that $\range H=\aff(\cC(\bfx^*))$.
Then, the mapping ${\rS}_{\rm GE}$ in \eqref{sge} has the Aubin property at the origin for $\bfx^*$ if and only if for all $\bfnu\in\Re^{\ell}\setminus\{\bfzero\}$,
\[
\label{aubr}
H^\top \Theta H\bfnu \notin -H^\top\left\{\sum_{j\in J_{4,5,6}}\nabla g^j(\bfx^*)\bfp^j \Big\vert\
\bfp^j\in \cD^*\cN_{\cQ_j}(g^j(\bfx^*),-(\bfmu^*)^j)(\cJ g^j(\bfx^*)H\bfnu) \right\}.
\]
\end{proposition}
\begin{proof}
In this setting, we know from Lemma \ref{aubproof2} that the mapping ${\rS}_{\rm GE}$ \eqref{sge} has the Aubin property at the origin for $\bfx^*$ if and only if \eqref{aubinreduce1} holds.
For a given vector $\bfd\in\aff(\cC(\bfx^*))\backslash\{\bfzero\}$, one can get from \eqref{affcricdef2} that  $\cJ g^j(\bfx^*)\bfd=\bfzero$ for all $j\in J_1$.
Then by Part (1) of Lemma \ref{codersecond} (or \eqref{coderzerofirst} if $r_j=0$),
$D^*\cN_{\cQ_j}(g^j(\bfx^*),-(\bfmu^*)^j)(\cJ g^j(\bfx^*)\bfd)=\Re^{1+r_j}$ holds for all $j\in J_1$.
Meanwhile, one has that
$D^*\cN_{\cQ_j}(g^j(\bfx^*),-(\bfmu^*)^j)(\cJ g^j(\bfx^*)\bfd)=\{\bfzero\}$ holds for all $j\in J_2$
by Part (2) of Lemma \ref{codersecond}
(or \eqref{coderzerofirst} if $r_j=0$).
Furthermore, if $j\in J_3$, one has
$\langle \cJ g^j(\bfx^*)\bfd, (\bfmu^*)^j\rangle =0$.
Then by Part (3) of Lemma \ref{codersecond} it holds that
$$
\begin{array}{ll}
\cD^*\cN_{\cQ_j}(g^j(\bfx^*),-(\bfmu^*)^j)(\cJ g^j(\bfx^*)\bfd)
\\
\qquad =
\left\{\left(-\frac{(\dot\mu^*)^j}{\dot g^j(\bfx^*)}{\dot\varepsilon^j}-\tau;\,
\frac{(\dot\mu^*)^j}{\dot g^j(\bfx^*)}\bar{\boldsymbol\varepsilon}^j+\tau\frac{\bar g^j(\bfx^*)}{\|\bar g^j(\bfx^*)\|}\right)
\mid
\tau\in\Re
\right\},
\quad\forall j\in J_3,
\end{array}
$$
where $\boldsymbol{\varepsilon}^j:= \cJ g^j(\bfx^*)\bfd$.
Note that
$$
\begin{array}{ll}
\nabla g^j(\bfx^*)\left(-\frac{(\dot\mu^*)^j}{\dot g^j(\bfx^*)}{\dot\varepsilon^j}-\tau;\,
\frac{(\dot\mu^*)^j}{\dot g^j(\bfx^*)}\bar{\boldsymbol\varepsilon}^j+\tau\frac{\bar g^j(\bfx^*)}{\|\bar g^j(\bfx^*)\|}\right)
\\
=-\frac{(\dot\mu^*)^j}{\dot g^j(\bfx^*)}\nabla g^j(\bfx^*)
\begin{pmatrix} 1&\bfzero^\top \\ \bfzero & -I_{r_j} \end{pmatrix}(\cJ g^j(\bfx^*)\bfd)
-\tau \nabla g^j(\bfx^*)
\left(1; - \frac{\bar g^j(\bfx^*)}{\|\bar g^j(\bfx^*)\|}\right),
\quad\forall j\in J_3.
\end{array}
$$
Therefore, by the above discussions and \eqref{ssocq}, the condition \eqref{aubinreduce1} can be equivalently written as for all $\bfd\in\aff(\cC(\bfx^*))\setminus\{\bfzero\}$ it holds that
\[
\label{aubin123}
\begin{array}{ll}
\Theta\bfd
\notin
\cR-\left\{
\sum\limits_{j\in J_{4,5,6}}\nabla g^j(\bfx^*)\bfp^j
\ \Big\vert \
\bfp^j\in \cD^*\cN_{\cQ_j}(g^j(\bfx^*),-(\bfmu^*)^j)(\cJ g^j(\bfx^*)
\bfd)
\right\},
\end{array}
\]
where $\cR$ is the subspace defined by
$$
\begin{array}{ll}
\cR:=
\Span\left(
\{\nabla h(\bfx^*) \}
\cup
\left\{ \nabla g^j(\bfx^*) \right\}_{j\in J_1}
\cup
\left\{\nabla g^j(\bfx^*)(\bfmu^*)^j \right\}_{j\in J_3}
\right).
\end{array}
$$
Note that $\range H=\aff(\cC(\bfx^*))$.
Hence, \eqref{aubin123} holds if and only if for all $\bfnu\in\Re^{\ell}\setminus\{\bfzero\}$ it holds that
\[
\label{aub321}
\begin{array}{ll}
\Theta H\bfnu
\notin
\cR-
\left\{
\sum\limits_{j\in J_{4,5,6}}\nabla g^j(\bfx^*)\bfp^j
\ \Big\vert \
\bfp^j\in \cD^*\cN_{\cQ_j}(g^j(\bfx^*),-(\bfmu^*)^j)(\cJ g^j(\bfx^*)H\bfnu)
\right\}.
\end{array}
\]
Therefore, the proof of the proposition is finished if \eqref{aubr} and \eqref{aub321} are equivalent.

On the one hand, suppose that \eqref{aubr} does not hold, i.e., there exists a nonzero vector $\tilde\bfnu\in\Re^\ell$ such that
\[
\label{aubf}
H^\top \Theta H \tilde \bfnu
=-H^\top \sum_{j\in J_{4,5,6}} \nabla g^j(\bfx^*)\bfp^j
\ \mbox{with} \
\bfp^j\in \cD^*\cN_{\cQ_j}(g^j(\bfx^*),-(\bfmu^*)^j)(\cJ g^j(\bfx^*)H\tilde\bfnu).
\]
Therefore, one has
$\Theta H\tilde \bfnu + \sum_{j\in J_{4,5,6}} \nabla g^j(\bfx^*)\bfp^j
\in \ker H^\top $.
Let $\Xi$ be the matrix defined in \eqref{defXi}. Then, for any $\bfzeta^0\in \Re^m$,
$\{\bfzeta^j\in \Re^{1+r_j}\}_{j\in J_1}$,
$\{\zeta^j\in \Re\}_{j\in J_3}$,
and
$\{\hat\bfzeta^j\in \Re^{r_j}\}_{j\in J_4}$
one has
$$
\begin{array}{l}
\ds
\Xi^\top(\bfzeta^0;\{\bfzeta^j\}_{j\in J_1};\{\zeta^j\}_{j\in J_3};\{\hat\bfzeta^j\}_{j\in J_4})
\\[2mm]
=\nabla h(\bfx^*)\bfzeta^0
+\sum_{j\in J_1}
\nabla g^{j}(\bfx^*)\bfzeta^j
+\sum_{j\in J_3}
\zeta^j\nabla g^{j}(\bfx^*)(\bfmu^*)^j
+\sum_{j\in J_4}
\nabla g^{j}(\bfx^*)S_j\hat\bfzeta^j.
\end{array}
$$
Since $\ker(H^\top )={\aff(\cC(\bfx^*))}^\perp=\range \Xi^\top$, by defining
$\overline \cR:=\cR\cup\Span (\{\nabla g^{j}(\bfx^*)S_j\}_{j\in J_4})$ one can get
$$
\begin{array}{ll}
\Theta H\tilde \bfnu
\in \overline \cR-
\left\{
\sum\limits_{j\in J_{4,5,6}}\nabla g^j(\bfx^*)\bfp^j
\ \Big\vert \
\bfp^j\in \cD^*\cN_{\cQ_j}(g^j(\bfx^*),-(\bfmu^*)^j)(\cJ g^j(\bfx^*)H\tilde \bfnu)
\right\}.
\end{array}
$$
Note that for any $j\in J_4$
one has from  \eqref{jdecompose} that
$(\bfmu^*)^j\in\boundary\cQ_j\setminus\{\bfzero\}$.
Meanwhile, as $H\tilde\bfnu\in \aff(\cC(\bfx^*))$, one has from \eqref{affcricdef2} that
$$
\begin{array}{ll}
\cJ g^j(\bfx^*)H\tilde \bfnu
=
\tilde \tau_j((\dot\mu^*)^j; -(\bar\bfmu^*)^j)
=
\tilde\tau_j (\dot\mu^*)^j\left(1; \frac{-(\bar\bfmu^*)^j}{(\dot\mu^*)^j}\right)
=
\tilde\tau_j (\dot\mu^*)^j\left(1; -\frac{(\bar\bfmu^*)^j}{\|(\bar\bfmu^*)^j\|}\right)
\end{array}
$$
with $\tilde \tau_j\in\Re$.
Recall that the columns of $S_j$ span a basis of $\{((\dot\mu^*)^j;-(\bar\bfmu^*)^j)\}^\perp$.
Consequently, we know that
$\langle S_j\hat \bfzeta^j, \cJ g^j(\bfx^*)H\tilde \bfnu\rangle=0$ for any $\hat\bfzeta^j\in \Re^{r_j}$.
Then by Part (5) of Lemma \ref{codersecond} one has
$$
-S_j\hat \bfzeta^j\in \cD^*\cN_{\cQ_j}(g^j(\bfx^*),-(\bfmu^*)^j)(\cJ g^j(\bfx^*)H\tilde \bfnu),
\quad\forall \hat\bfzeta^j\in \Re^{r_j},
\quad\forall j\in J_4.
$$
Recall from \eqref{aubf} that
$\bfp^j\in \cD^*\cN_{\cQ_j}(g^j(\bfx^*),-(\bfmu^*)^j)(\cJ g^j(\bfx^*)H\tilde\bfnu)$.
Meanwhile, for any $j\in J_4$,
it holds that $\langle \bfp^j- S_j\hat \bfzeta^j, \cJ g^j(\bfx^*)H\tilde\bfnu\rangle =\langle \bfp^j, \cJ g^j(\bfx^*)H\tilde\bfnu\rangle$.
Then, it is easy to see from Part (5) of Lemma \ref{codersecond} that
$$
\bfp^j- S_j\hat \bfzeta^j\in
\cD^*\cN_{\cQ_j}(g^j(\bfx^*),-(\bfmu^*)^j)(\cJ g^j(\bfx^*)H\tilde\bfnu),
\quad\forall \hat\bfzeta^j\in \Re^{r_j},
\quad\forall j\in J_4.
$$
Consequently, \eqref{aub321} does not hold.

On the other hand, suppose that \eqref{aub321}  is not true.
Then there are vectors and real numbers
$\hat\bfnu\in\Re^{\ell}$,
$\bfzeta^0\in \Re^m$,
$\{\bfzeta^j\in \Re^{1+r_j}\}_{j\in J_1}$,
$\{\zeta^j\in \Re\}_{j\in J_3}$,
and
$\{\bfp^j\in \Re^{1+r_j}\}_{j\in J_{4,5,6}}$
such that
$$
\Theta H\hat\bfnu
=
\nabla h(\bfx^*)  \bfzeta^0
+ \sum_{j\in J_1}\nabla g^{j}(\bfx^*) \bfzeta^j
+ \sum_{j\in J_3} \zeta^j\nabla g^j(\bfx^*)(\bfmu^*)^j
- \sum_{j\in J_{4,5,6}}\nabla g^j(\bfx^*)\bfp^j
$$
with
$\bfp^j\in \cD^*\cN_{\cQ_j}(g^j(\bfx^*),-(\bfmu^*)^j)(\cJ g^j(\bfx^*)H\hat\bfnu)$ for all $j\in J_{4,5,6}$.
Then, for any $\bfnu\in\Re^{\ell}$, by noting that $H\bfnu\in\aff(\cC(\bfx^*))$ and using \eqref{affcricdef2} one can get
$$
\begin{array}{ll}
\langle \bfnu, H^\top\Theta H\hat\bfnu\rangle
=
\langle H\bfnu, \Theta H\hat\bfnu\rangle
\\
=
\langle \cJ h(\bfx^*)H\bfnu, \bfzeta^0\rangle
+\sum\limits_{j\in J_1} \langle \cJ g_{j}(\bfx^*)H\bfnu,  \bfzeta^j\rangle
+\sum\limits_{j\in J_3} \zeta^j\langle \cJ g^{j}(\bfx^*)H\bfnu,(\bfmu^*)^j\rangle
\\
\quad
-\sum\limits_{j\in J_{4,5,6}}\langle H\bfnu,\nabla g^j(\bfx^*)\bfp^j\rangle
=
-\sum\limits_{j\in J_{4,5,6}}\langle \bfnu, H^\top\nabla g^j(\bfx^*)\bfp^j\rangle.
\end{array}
$$
Since $\bfnu$ can be arbitrarily chosen, one can get
$H^\top\Theta H\hat\bfnu=
-\sum_{j\in J_{4,5,6}}
 H^\top\nabla g^j(\bfx^*)\bfp^j$, which contradicts \eqref{aubr}.
This completes the proof.
\end{proof}

The following result is a direct consequence of Proposition \ref{aubins}.

\begin{corollary}
\label{aubinss}
Suppose that Assumption \ref{assblanket} holds.
Let $\Theta\in\Re^{n\times n}$  be the matrix given in \eqref{ssocq},
and $H\in\Re^{n\times \ell}$ be a matrix with full column rank such that $\range H=\aff(\cC(\bfx^*))$.
Then the mapping ${\rS}_{\rm GE}$ in \eqref{sge} having the Aubin property at the origin for $\bfx^*$ implies that for all
$\bfnu\in\Re^\ell\setminus\{\bfzero\}$ it holds that
\[
\label{aubrxnew}
\begin{array}{ll}
H^\top \Theta H\bfnu
\notin
-
\left\{
\sum\limits_{j\in J_{4,5}} \cG_j^\top p^j
+\sum\limits_{j\in J_6} \cG_j^\top \bfp^j \ \Big\vert\
\begin{array}{ll}
p^j\in \cD^*\cN_{[0,+\infty)}(0,0) (\cG_j\bfnu),
\\[1mm]
\bfp^j\in \cD^*\cN_{\cQ_j}(\bfzero,\bfzero)(\cG_j\bfnu)
\end{array}
\right\},
\end{array}
\]
where each $\cG_j$ is either a vector in $\Re^{\ell}$
or a matrix in $\Re^{(1+r_j)\times\ell}$ defined by
\[
\label{defgj}
\cG_j:=
\begin{cases}
((\dot\mu^*)^j;-(\bar\bfmu^*)^j)^\top
\cJ g^j(\bfx^*)H,
& \forall \, j\in J_4,
\\
(\dot g^j(\bfx^*);-\bar g^j(\bfx^*))^\top
\cJ g^j(\bfx^*)H,
& \forall \, j\in J_5,
\\
\cJ g^j(\bfx^*)H, & \forall \, j\in J_6.
\end{cases}
\]
Moreover, the matrix $H^\top \Theta H\in\Re^{\ell\times\ell}$ is nonsingular.
\end{corollary}
\begin{proof}
Let $\bfzero\neq \bfnu\in\Re^\ell$ be arbitrarily chosen.
For $j\in J_4$,
one has from \eqref{affcricdef2} that $\cJ g^j(\bfx^*)H\bfnu=\tau_j((\dot\mu^*)^j; -(\bar\bfmu^*)^j)$ with $\tau_j\in\Re$.
Meanwhile, from \eqref{jdecompose} one has $g^j(\bfx^*)=\bfzero$ and $(\bfmu^*)^j\in\boundary\cQ_j\setminus\{\bfzero\}$.
Then by Part (5) of Lemma \ref{codersecond} and \eqref{defgj} one has
$\bfp^j \in \cD^*\cN_{\cQ_j}(g^j(\bfx^*),-(\bfmu^*)^j)(\cJ g^j(\bfx^*)H\bfnu)$
if and only if
$$
\left\{
\begin{array}{ll}
\bfp^j\in\Re^{1+r_j}, &
\mbox{ if }\, \cG_j\bfnu=0,
\\
\langle \cJ g^j(\bfx^*)H\bfnu, \bfp^j\rangle=0, &
\mbox{ if }\, \cG_j\bfnu>0,
\\
\langle \cJ g^j(\bfx^*)H\bfnu, \bfp^j\rangle \ge 0, &
\mbox{ if }\, \cG_j\bfnu<0,
\end{array}
\right.
$$
which, by \eqref{coderzero}, is equivalent to
$- \langle \cJ g^j(\bfx^*)H\bfnu, \bfp^j\rangle \in\cD^*\cN_{[0,+\infty)}(0,0) (\cG_j\bfnu)$.

For $j\in J_5$, from \eqref{jdecompose}
one has $g^j(\bfx^*)\in\boundary\cQ_j\setminus\{\bfzero\}$ and $(\bfmu^*)^j=\bfzero$.
Therefore,
by Part (4) of Lemma \ref{codersecond} one has
$\bfp^j \in \cD^*\cN_{\cQ_j}(g^j(\bfx^*),-(\bfmu^*)^j)(\cJ g^j(\bfx^*)H\bfnu)$
if and only if $\bfp^j=p^j (\dot g^j(\bfx^*);-\bar g^j(\bfx^*))$ with
$$
p^j\in \Re, \mbox{ if }\, \cG_j\bfnu= 0,
\qquad
p^j =0, \mbox{ if }\, \cG_j\bfnu > 0,
\qquad
\mbox{and}\quad
p^j \le 0, \mbox{ if }\, \cG_j\bfnu< 0.
$$
By \eqref{coderzero}, the above condition on $p^j$ is exactly $p^j\in \cD^*\cN_{[0,+\infty)}(0,0) (\cG_j\bfnu)$.
Furthermore, one has $g^j(\bfx^*)=(\bfmu^*)^j=\bfzero$ for $j\in J_6$.
Therefore, \eqref{aubr} in Lemma \ref{aubins} is equivalent to
\[
\label{aubrx}
\begin{array}{ll}
H^\top \Theta H\bfnu
\notin
-
\left\{
\begin{array}{ll}
 \sum\limits_{j\in J_4} H^\top\nabla g^j(\bfx^*)\bfp^j
+\sum\limits_{j\in J_5} \cG_j^\top p^j
+\sum\limits_{j\in J_6} \cG_j^\top \bfp^j
\ \Big\vert\
\qquad\qquad
\\[3mm] \hfill
-\langle  \cJ g^j(\bfx^*)H\bfnu,\bfp^j \rangle  \in\cD^*\cN_{[0,+\infty)}(0,0) (\cG_j \bfnu),\ j\in J_4,
\\[2mm]
\hfill
p^j\in \cD^*\cN_{[0,+\infty)}(0,0) (\cG_j\bfnu),\ j\in J_5,
\\[2mm] \hfill
\bfp^j\in \cD^*\cN_{\cQ_j}(\bfzero,\bfzero)(\cG_j\bfnu),\ j\in J_6
\end{array}
\right\}.
\end{array}
\]
Note that for $j\in J_4$, one has
$$
\begin{array}{ll}
\left\{
\cG_j^\top p^j\mid p^j\in\cD^*\cN_{[0,+\infty)}(0,0) (\cG_j\bfnu)
\right\}
\\[2mm]
=\left\{
\cG_j^\top p^j
\
\Bigg\vert
\
\begin{array}{ll}
p^j\in \Re, &
\mbox{ if }\, ((\dot\mu^*)^j;-(\bar\bfmu^*)^j)^\top
\cJ g^j(\bfx^*)H \bfnu=0,
\\
p^j=0, &
\mbox{ if }\, ((\dot\mu^*)^j;-(\bar\bfmu^*)^j)^\top
\cJ g^j(\bfx^*)H \bfnu>0,
\\
p^j \le 0, &
\mbox{ if }\, ((\dot\mu^*)^j;-(\bar\bfmu^*)^j)^\top
\cJ g^j(\bfx^*)H \bfnu<0
\end{array}
\right\}
\\[5mm]
\subseteq
\left\{H^\top\nabla  g^j(\bfx^*) \bfp^j
\
\Big\vert
\
\begin{array}{ll}
-\bfp^j\in\Re^{1+r_j}, &
\mbox{ if }\, \cG_j\bfnu=0,
\\
-\langle \cJ g^j(\bfx^*)H\bfnu, \bfp^j\rangle=0, &
\mbox{ if }\, \cG_j\bfnu>0,
\\
-\langle \cJ g^j(\bfx^*)H\bfnu, \bfp^j\rangle \le 0, &
\mbox{ if }\, \cG_j\bfnu<0
\end{array}
\right\}
\\[5mm]
=
\left\{H^\top\nabla  g^j(\bfx^*) \bfp^j
\ \Big\vert\
-\langle  \cJ g^j(\bfx^*)H\bfnu,\bfp^j \rangle  \in\cD^*\cN_{[0,+\infty)}(0,0) (\cG_j \bfnu)
\right\}.
\end{array}
$$
Therefore, \eqref{aubrxnew} is a consequence of \eqref{aubrx}.

Suppose that there exists a nonzero vector $\bfnu\in\Re^l$ such that $H^\top \Theta H\bfnu=\bfzero$.
In this case, one can take $\bfp^j=\bfzero$ for all $j\in J_{4,6}$ and $p^j=0$ for all $j\in J_5$ to get valid coderivatives in \eqref{aubrx}
via \eqref{coderzero} and Part (6e) of Lemma \ref{codersecond}.
This contradicts \eqref{aubrx},
so the matrix $H^\top \Theta H$ is nonsingular,
which completes the proof.
\end{proof}

We have the following result for supplementing Corollary \ref{aubinss}.
\begin{corollary}
\label{aubins2}
Suppose that Assumption \ref{assblanket} holds.
Let $\Theta\in\Re^{n\times n}$  be the matrix defined in \eqref{ssocq}, and  $H\in\Re^{n\times \ell}$ be a matrix with full column rank such that $\range H=\aff(\cC(\bfx^*))$.
For $\cG_j$ defined in \eqref{defgj}, one has
\[
\label{fullrange}
\range\left(\{\cG_j\}_{{j\in J_{4,5,6}}}\right)
=
\Re^{|J_{4,5}|
+\sum_{j\in J_6}(1+r_j)},
\]
where
$\{\cG_j\}_{{j\in J_{4,5,6}}}$ is the matrix with rows being $\cG_j$.

\end{corollary}
\begin{proof}
For the given vectors $\zeta^j\in \Re$ ($j\in J_{4,5}$) and $\bfzeta^j\in\Re^{1+r_j}$ ($j\in  J_6$),
we show that there exists a vector $\bfnu\in\Re^\ell$ such that $\cG_j\bfnu =\zeta^j$ for all $j\in J_{4,5}$
and $\cG_j \bfnu =\bfzeta^j$ for all $j\in J_6$.
Recall that the constraint nondegeneracy condition \eqref{cnd} holds at $\bfx^*$.
Then, for the given vectors $\bfz^j\in\Re^{1+r_j}$ defined by
$$
\begin{array}{lllll}
\bfz^j:=\bfzero, &\mbox{ if } j\in J_{1,2,3},
&\quad &
\bfz^j:=\frac{\zeta^j((\dot\mu^*)^j;-(\bar\bfmu^*)^j)}{\|(\bfmu^*)^j\|^2},
&\mbox{ if } j\in J_4,
\\
\bfz^j:=\bfzeta^j, &\mbox{ if } j\in J_6,
&\mbox{and}&
\bfz^j:=\frac{\zeta^j(\dot g^j(\bfx^*); -\bar g^j(\bfx^*))}{\|g^j(\bfx^*)\|^2},
&\mbox{ if } j\in J_5,
\end{array}
$$
there always exists a vector
$\bfd\in\Re^n$ such that
$$
\cJ h (\bfx^*)\bfd=\bfzero
\ \mbox{and}\
\cJ g^{j}(\bfx^*)\bfd+\bfv^j=\bfz^j
\ \mbox{with}
\
\bfv^j\in
\lin\left(\cT_{\cQ_j}(g^j(\bfx^*))\right), \quad\forall j=1,\dots, J,
$$
where $\lin\left(\cT_{\cQ_j}(g^j(\bfx^*))\right)$ is defined in \eqref{lintan}.
It can be routinely examined from \eqref{affcricdef2} that $\bfd\in\aff(\cC(\bfx^*))$.
Moreover, since $H$ has full column rank such that $\range H=\aff(\cC(\bfx^*))$, there exists a unique vector $\bfnu\in\Re^\ell$ such that $H\bfnu=\bfd$.
For all $j\in J_4$, one has $\bfv^j=\bfzero$, so that by \eqref{defgj} one can get
$\cG_j\bfnu=((\dot\mu^*)^j;-(\bar\bfmu^*)^j)^\top
\cJ g^j(\bfx^*)H\bfnu=\zeta^j$.
Furthermore, for $j\in J_5$, one has $\langle \bar\bfv^j,\bar g^j(\bfx^*)\rangle-\dot{v}^j\dot g^j(\bfx^*) = 0$.
Thus by \eqref{defgj} one has
$$
\cG_j\bfnu=(\dot g^j(\bfx^*);-\bar g^j(\bfx^*))^\top
\cJ g^j(\bfx^*)H \bfnu = (\dot g^j(\bfx^*);-\bar g^j(\bfx^*))^\top(\bfz^j-\bfv^j)=\zeta^j.
$$
Finally, for $j\in J_6$, one has from \eqref{defgj} that $\bfv^j=\bfzero$ and $\cG_j\bfnu=\cJ g^j(\bfx^*)H\bfnu=\bfzeta^j$.
Consequently, we know that \eqref{fullrange} holds,
which completes the proof.
\end{proof}

Note that in the reduction procedure introduced above,
we have reduced the original condition \eqref{aubinoriginal} for characterizing the Aubin property of  $\rS_{\rm GE}$  to the condition \eqref{aubrxnew} in Corollary \ref{aubinss}.
The former condition is related to all the constraints in the nonlinear SOCP \eqref{nlsocp},
while the latter condition is only related to the blocks at which the strict complementarity does not hold.

\subsection{Deriving the strong second-order sufficient condition}
Here we reformulate the second-order optimality conditions \eqref{sosc} and \eqref{ssosc} to
fit \eqref{aubrxnew} in Corollary \ref{aubinss}.

\begin{lemma}
\label{lemsosc}
Suppose that Assumption \ref{assblanket} holds.
Let $\Theta\in\Re^{n\times n}$  be the matrix given in \eqref{ssocq}, and  $H\in\Re^{n\times \ell}$ be a matrix with full column rank such that $\range H=\aff(\cC(\bfx^*))$.
Then the second-order sufficient condition \eqref{sosc} is equivalent to
\[
\label{soscreformulation}
\langle \bfnu, H^\top \Theta H \bfnu \rangle>0, \quad \forall \, \bfnu\in\cV\setminus\{\bfzero\},
\]
where $\cV$ is the closed convex cone defined by
\[\label{vg}
\cV=\left\{ \bfnu\in\Re^\ell \, \middle| \,
\cG_j \bfnu\ge 0,
\, \forall \, j\in J_{4,5},
\
\cG_j \bfnu \in \cQ_j,
\, \forall \, j\in J_6
\right\}
\]
with each $\cG_j$ being defined in \eqref{defgj}.
Moreover, the strong second-order sufficient condition \eqref{ssosc} is equivalent to
$H^\top \Theta H \succ O$.
\end{lemma}
\begin{proof}
According to the definition of $\Theta$ in \eqref{ssocq}, one can rewrite \eqref{sosc} as
\[
\label{sosc2}
\langle \bfd,\Theta\bfd\rangle >0,
\quad
\forall\, \bfd\in \cC(\bfx^*)\setminus\{\bfzero\},
\]
where the critical cone $\cC(\bfx^*)$ is given by \eqref{criticalcone2}.
Since $H\in\Re^{n\times\ell}$ is of full column rank and $\range H=\aff(\cC(\bfx^*))$,
from the definition of the critical cone in \eqref{criticalcone2}
and its affine hull in \eqref{affcricdef2},
one can see that the second-order sufficient condition \eqref{sosc2} is equivalent to \eqref{soscreformulation} with
$$
\cV:=\left\{ \bfnu\in\Re^\ell \,\middle|\,
\begin{array}{ll}
\cJ g^j(\bfx^*)H\bfnu \in
\Re_+ (({\dot\mu^*})^j;-({\bar\bfmu^*})^j), &\forall\, j\in J_4,
\\[1mm]
\cJ g^j(\bfx^*)H\bfnu \in \cT_{\cQ_j}(g^j(\bfx^*)), & \forall\, j\in J_{5,6}
\end{array}
\right\}.
$$
Based on the explicit formulations of the tangent cones in \eqref{tancone}, one can reformulate $\cV$ as
the one in \eqref{vg}.
The conclusion on the strong second-order sufficient condition directly comes from the fact that $\range H=\aff(\cC(\bfx^*))$.
This completes the proof.
\end{proof}

Let $\Theta\in\Re^{n\times n}$ be the matrix given in \eqref{ssocq}. Suppose that Assumption \ref{assblanket} holds and that $H\in\Re^{n\times \ell}$ is a matrix with full column rank such that $\range H=\aff(\cC(\bfx^*))$.
To prove the strong second-order sufficient condition, i.e., $H^\top \Theta H \succ O$, we construct a series of symmetric matrices of different dimensions and prove that they are positive definite.

We further suppose that \eqref{aubrxnew} holds with $\cG_j$ ($j\in J_{4,5,6}$) being defined by \eqref{defgj}.
For convenience,  for each integer $i\in\{1,\ldots, |J_{4,5,6}|\}$,
we assign it a unique index $j\in J_{4,5,6}$ denoted by $j_i$.
Without loss of generality, we assume that
\[
\label{jirule}
j_1,\ldots,j_{|J_4|}\subseteq J_4
\quad\mbox{and}
\quad
j_{|J_4|+1},\ldots,j_{|J_{4,5}|}\subseteq J_5.
\]
Then one can write \eqref{defgj} as follows
\[
\label{defginew}
\cG_{j_i}:=
\begin{cases}
((\dot\mu^*)^{j_i};-(\bar\bfmu^*)^{j_i})^\top\cJ g^{j_i}(\bfx^*)H,
& i=1,\ldots,|J_4|,
\\
(\dot g^{j_i}(\bfx^*);-\bar g^{j_i}(\bfx^*))^\top
\cJ g^{j_i}(\bfx^*)H,
& i=|J_4|+1,\ldots,|J_{4,5}|,
\\
\cJ g^{j_i}(\bfx^*)H, &   i=|J_{4,5}|+1, \ldots, |J_{4,5,6}|.
\end{cases}
\]
Recall from Corollary \ref{aubinss} that $H^\top\Theta H$ is nonsingular and \eqref{fullrange} holds.
Based on \eqref{defginew}, we can define the following matrices recursively:
\[
\label{defr}
\begin{cases}
R_i\in\Re^{(\ell-i+1)\times(\ell-i)},
& i=1,\ldots, |J_{4,5}|,
\\
R_i\in
\Re^{\left(\ell-i+1-\sum_{k=|J_{4,5}|+1}^{i-1}r_{j_k}\right)
\times
\left(\ell-i-\sum_{k=|J_{4,5}|+1}^{i}r_{j_k}\right)},
& i=|J_{4,5}|+1,\, \ldots\, , |J_{4,5,6}|
\end{cases}
\]
such that
\[
\label{rrlation}
\range\, R_i=\ker\, (\cG_{j_i} P_{i-1}),
\]
where $P_0\in \Re^{\ell\times\ell}$ is the identity matrix and
\[
\label{defp}
\begin{cases}
P_i:=R_1\times R_2\times\cdots\times R_i\in\Re^{\ell\times(\ell-i)},
&i=1,\ldots, |J_{4,5}|,
\\
P_i:= R_1 \times R_2 \times \cdots \times R_i \in \Re^{\ell \times
\left(\ell-i-\sum_{k=|J_{4,5}|+1}^{i}r_{j_k}\right)},
&i=|J_{4,5}|+1,\ldots, |J_{4,5,6}|.
\end{cases}
\]
In addition, define the matrices $T_0:=H^\top \Theta H$ and
\[
\label{defti}
T_i:=P_i^\top H^\top \Theta H P_i=
R_i^\top\times\cdots\times R_1^\top T_0
R_1\times\cdots\times R_i,
\quad
i=1,\ldots, |J_{4,5,6}|.
\]
It is obvious from \eqref{rrlation} that each $P_i$ has full column rank.
Moreover, we have the following result, which is the most crucial step in our analysis.

\begin{proposition}
\label{propreduce}
Suppose that Assumption \ref{assblanket} holds.
Let $\Theta\in\Re^{n\times n}$  be the matrix given in \eqref{ssocq} and $H\in\Re^{n\times \ell}$ be a matrix with full column rank such that $\range H=\aff(\cC(\bfx^*))$.
Suppose that \eqref{aubrxnew} holds.
For the matrices defined in \eqref{defr}, \eqref{defp} and \eqref{defti}, one has that
each $T_i$ is nonsingular, and
\[
\label{positivequad}
 [\cG_{j_i}P_{i-1}] T_{i-1}^{-1}  [\cG_{j_i}P_{i-1}]^\top \succ O,
\quad
\forall\, i=1,\ldots, |J_{4,5,6}|.
\]
\end{proposition}

\begin{proof}

Since \eqref{aubrxnew} holds in this setting, by using \eqref{jirule} we can get that, for any
$\bfnu\in\Re^\ell\setminus\{\bfzero\}$,
\[
\label{ordered}
\begin{array}{ll}
H^\top \Theta H\bfnu
\notin
-
\left\{
\sum\limits_{i=1}^{|J_{4,5}|}
\cG_{j_i}^\top p^{j_i}
+\sum\limits_{i= |J_{4,5}|+1}^{|J_{4,5,6}|}
\cG_{j_i}^\top \bfp^{j_i} \ \Big\vert\
\begin{array}{ll}
p^{j_i}\in \cD^*\cN_{[0,+\infty)}(0,0) (\cG_{j_i}\bfnu),
\\
\bfp^{j_i}\in \cD^*\cN_{\cQ_{j_i}}(\bfzero,\bfzero)(\cG_{j_i}\bfnu)
\end{array}
\right\}.
\end{array}
\]
We first show that for any $k\in\{0,\ldots, |J_{4,5}|\}$ and any $\bfeta \in\Re^{\ell-k}\setminus\{\bfzero\}$,
it holds that
\[
\label{orderedproduction}
\begin{array}{ll}
T_k \bfeta
\notin -
P_k^{\top} \left\{
\sum\limits_{i=k+1}^{|J_{4,5}|}
\cG_{j_i}^\top p^{j_i}
+\sum\limits_{i= |J_{4,5}|+1}^{|J_{4,5,6}|}
\cG_{j_i}^\top \bfp^{j_i}
\Big\vert
\begin{array}{ll}
p^{j_i}\in \cD^*\cN_{[0,+\infty)}(0,0) (\cG_{j_i}P_k \bfeta),
\\[1mm]
\bfp^{j_i}\in \cD^*\cN_{\cQ_{j_i}}(\bfzero,\bfzero)(\cG_{j_i}P_k \bfeta)
\end{array}
\right\}.
\end{array}
\]
Recall that $T_0=H^\top\Theta H$.
Then, by \eqref{ordered} we know that \eqref{orderedproduction} holds for $k=0$.
If $1\le k\le |J_{4,5}|$, one has from \eqref{defp} that
$\cG_{j_i}P_k =\cG_{j_i} R_1\times R_2\times\cdots\times R_k$.
Then one can get from \eqref{rrlation} that
$\cG_{j_i}P_k\bfeta=0$ for any $\bfeta\in\Re^{\ell-k}$ and all $i=1,\ldots, k$.
Thus, by \eqref{coderzero} we know that \eqref{ordered} implies that,
for any $\bfeta \in\Re^{\ell-k}\setminus\{\bfzero\}$,
it holds that
\[
\label{orderedreduced}
\begin{array}{ll}
H^\top \Theta H P_k \bfeta
\notin
\Span\left(\{\cG_{j_i}^\top\}_{i=1,\ldots, k}\right)
\\[2mm]
\qquad-
\left\{
\sum\limits_{i=k+1}^{|J_{4,5}|}
\cG_{j_i}^\top p^{j_i}
+\sum\limits_{i= |J_{4,5}|+1}^{|J_{4,5,6}|}
\cG_{j_i}^\top \bfp^{j_i} \ \Big\vert\
\begin{array}{ll}
p^{j_i}\in \cD^*\cN_{[0,+\infty)}(0,0) (\cG_{j_i}P_k \bfeta),
\\
\bfp^{j_i}\in \cD^*\cN_{\cQ_{j_i}}(\bfzero,\bfzero)(\cG_{j_i}P_k \bfeta)
\end{array}
\right\}.
\end{array}
\]
Note that
$P_k^\top\cG_{j_i}^\top=\bfzero$ for all $i=1,\ldots, k$,
so that
$\{\cG_{j_i}^\top\}_{i=1,\ldots, k}\subseteq \ker P_k^\top$.
Moreover, as $P_k^\top\in \Re^{(l-k)\times l}$ has full row rank by definition, the dimension of $\ker P_k^\top$ is exactly $k$.
Then by Corollary \ref{aubins2} we know that
$\Span(\{\cG_{j_i}^\top\}_{i=1,\ldots, k})= \ker P_k^\top$.
Suppose to the contrary that \eqref{orderedproduction} can not be true.
That is there exist a nonzero vector $\hat\bfeta\in\Re^{\ell-k}$,
scalars $\hat p^{j_i}\in\Re$ for all $i=k+1,\ldots,|J_{4,5}|$,
and vectors $\hat\bfp^{j_i}\in\Re^{1+r_{j_i}}$  for all $i=|J_{4,5}|+1, \ldots, |J_{4,5,6}|$ satisfying
$\hat p^{j_i}\in\cD^*\cN_{[0,+\infty)}(0,0) (\cG_{j_i}P_k \hat\bfeta)$ and
$\hat\bfp^{j_i}\in \cD^*\cN_{\cQ_{j_i}}(\bfzero,\bfzero)(\cG_{j_i}P_k \hat\bfeta)$,
such that
$$
\begin{array}{ll}
T_k \hat\bfeta
= P_k^\top H^\top \Theta H P_k \hat \bfeta
=-P_k^\top \Big(
\sum\limits_{i=k+1}^{|J_{4,5}|}
\cG_{j_i}^\top \hat  p^{j_i}
+\sum\limits_{i= |J_{4,5}|+1}^{|J_{4,5,6}|}
\cG_{j_i}^\top \hat \bfp^{j_i}
\Big).
\end{array}
$$
Then one has
$$
\begin{array}{ll}
H^\top\Theta H P_k \hat\bfeta
+\sum\limits_{i=k+1}^{|J_{4,5}|} \cG_{j_i}^\top \hat p^{j_i}
+\sum\limits_{i= |J_{4,5}|+1}^{|J_{4,5,6}|} \cG_{j_i}^\top \hat \bfp^{j_i}
\in \ker P_k^\top
=\Span\left(\{\cG_{j_i}^\top\}_{i=1,\ldots, k}\right),
\end{array}
$$
which contradicts \eqref{orderedreduced}.
{Particularly, if  $T_k\hat\bfeta=\bfzero$, one can take $\hat p^{j_i}=0$ and $\hat\bfp^{j_i}=\bfzero$ due to
\eqref{coderzero} and Part (6e) of Lemma \ref{codersecond}.}
Therefore, we know that \eqref{orderedproduction} holds and $T_k$ is nonsingular.

Next, we show that for any $k\in\{|J_{4,5}|+1,\ldots, |J_{4,5,6}|\}$ one has
\[
\label{orderedproduction2}
\begin{array}{ll}
T_k \bfeta
\notin -
P_k^{\top} \left\{
\sum\limits_{i= k+1}^{|J_{4,5,6}|}
\cG_{j_i}^\top \bfp^{j_i}
\ \big\vert\
\bfp^{j_i}\in \cD^*\cN_{\cQ_{j_i}}(\bfzero,\bfzero)(\cG_{j_i}P_k \bfeta)
\right\},
\forall\, \bfeta \in\Re^{\ell-\tilde k}
\setminus\{\bfzero\},
\end{array}
\]
where $\tilde k:= k+\sum_{t=|J_{4,5}|+1}^k r_{j_t}$.
If $|J_{4,5}|+1 \le k\le |J_{4,5,6}|$,
from \eqref{defp} one has
$\cG_{j_i}P_k =\cG_{j_i} R_1\times R_2\times\cdots\times R_k$.
According to \eqref{rrlation}  one can get
$\cG_{j_i}P_k\bfeta=\bfzero$ for any $\bfeta\in\Re^{\ell-\tilde k}$ and  any $i=1,\ldots, k$.
Then, by \eqref{coderzero} and Part (6a) of Lemma \ref{codersecond}
(or \eqref{coderzero} if $r_j=0$)
we know that \eqref{ordered} implies that,
for any $\bfeta \in\Re^{\ell-\tilde k}\setminus\{\bfzero\}$,
\[
\label{orderedreduced2}
\begin{array}{ll}
H^\top \Theta H P_k \bfeta
\notin
\Span\left(\{\cG_{j_i}^\top\}_{i=1,\ldots, k}\right)
-
\left\{
\sum\limits_{i= k+1}^{|J_{4,5,6}|}
\cG_{j_i}^\top \bfp^{j_i} \ \Big\vert\
\bfp^{j_i}\in \cD^*\cN_{\cQ_{j_i}}(\bfzero,\bfzero)(\cG_{j_i}P_k \bfeta)
\right\}.
\end{array}
\]
Note that $P_k^\top\cG_{j_i}^\top=\bfzero$ for all $i=1,\ldots, |J_{4,5}|$
and $P_k^\top\cG_{j_i}^\top= O_{(\ell-\tilde k)\times(1+r_{j_i})}$ for all $i=|J_{4,5}|+1,\ldots, k$.
Moreover, since $P_k^\top\in \Re^{(l-\tilde k)\times l}$ has full row rank, one has
$$
\{\cG_{j_i}^\top\}_{i=1,\ldots, |J_{4,5}|}\cup\{\Span\cG_{j_i}^\top\}_{i=|J_{4,5}|+1,\ldots,k}\subseteq \ker P_k^\top.
$$
Then by Corollary \ref{aubins2} we know that
$\Span\left(\{\cG_{j_i}^\top\}_{i=1,\ldots, k}\right)= \ker P_k^\top$.
Suppose to the contrary that \eqref{orderedproduction2} is not true,
i.e.,
there exists a nonzero vector $\tilde \bfeta\in\Re^{\ell-\tilde k}$  and vectors $\tilde \bfp^{j_i}\in\Re^{1+r_{j_i}}$ ($i=k+1,\ldots,|J_{4,5,6}|$)
such that
$$
\begin{array}{ll}
T_k \tilde\bfeta
= P_k^\top H^\top \Theta H P_k \tilde \bfeta
=-P_k^\top \left(
\sum\limits_{i=k+1}^{|J_{4,5,6}|}
\cG_{j_i}^\top \tilde \bfp^{j_i}\right)
\ \mbox{and}\
\tilde\bfp^{j_i}\in \cD^*\cN_{\cQ_{j_i}}(\bfzero,\bfzero)(\cG_{j_i}P_k \tilde\bfeta).
\end{array}
$$
Then one can get
$$
\begin{array}{ll}
H^\top\Theta H P_k \tilde\bfeta
+\sum\limits_{i= k+1}^{|J_{4,5,6}|} \cG_{j_i}^\top \tilde \bfp^{j_i}
\in \ker (P_k^\top)
=\Span(\{\cG_{j_i}^\top\}_{i=1,\ldots, k}),
\end{array}
$$
which contradicts \eqref{orderedreduced2}.
Therefore,
\eqref{orderedproduction2} holds, and it is easy to see that $T_k$ is nonsingular.

In the following, we separate the proof of \eqref{positivequad} into three cases:

\smallskip
{\bf(1)}
Suppose that there exists an index $\iota\in\{1,\ldots, |J_4|\}$ such that
$$
[\cG_{j_{\iota}}P_{\iota-1}]
T_{\iota-1}^{-1}  [\cG_{j_{\iota}}P_{\iota-1}]^\top
\le 0.
$$
Let $k:=\iota-1$ in \eqref{orderedproduction}.
According to \eqref{coderzero} and Part (6e) of Lemma \ref{codersecond},
one can fix all $p^{j_i}=0$ and all $\bfp^{j_i}=\bfzero$ in \eqref{orderedproduction} for all $i\neq \iota$.
By doing this one can obtain that
\[
\label{reduced4}
T_{\iota-1}
\bfeta
\neq
-P_{\iota-1}^{\top}
\cG_{j_{\iota}}^\top p,
\quad
\forall \, p \in \cD^*\cN_{[0,+\infty)}(0,0) (\cG_{j_{\iota}}P_{\iota-1}\bfeta),
\quad
\forall\, \bfeta \in\Re^{\ell-\iota +1}
\setminus\{\bfzero\}.
\]
From \eqref{defginew} one has $\cG_{j_{\iota}}=((\dot\mu^*)^{j_{\iota}};-(\bar\bfmu^*)^{j_{\iota}})^\top\cJ g^{j_{\iota}}(\bfx^*)H$.
Then, by defining the vector
$\bfeta:=T_{\iota -1}^{-1}
[\cG_{j_{\iota}} P_{\iota-1}]^\top$ one can get
\[
\label{qlezero}
\langle
(\dot\mu^*)^{j_{\iota}};-(\bar\bfmu^*)^{j_{\iota}}),
\cJ g^{j_{\iota}}(\bfx^*)H P_{{\iota}-1}\bfeta\rangle
=
\cG_{j_{\iota}}P_{{\iota}-1}\bfeta
=
[\cG_{j_{\iota}}P_{{\iota}-1}] T_{{\iota}-1}^{-1}  [\cG_{j_{\iota}}P_{{\iota}-1}]^\top
\le 0.
\]
It can be easily observed from \eqref{coderzero} that
$p:=
-1
\in \cD^*\cN_{[0,+\infty)}(0,0)(\cG_{j_{\iota}}P_{{\iota}-1}\bfeta)$.
Meanwhile, one has
$T_{{\iota}-1}\bfeta=[\cG_{j_{\iota}}P_{{\iota}-1}]^\top
=-P_{{\iota}-1}^\top  \cG_{j_{\iota}}^\top p$.
However, this contradicts \eqref{reduced4}, so that
\eqref{qlezero} fails. Therefore, \eqref{positivequad} holds for all $i\in\{1,\ldots, |J_4|\}$.

\medskip
{\bf(2)}
Suppose that there exists an index $\iota\in\{|J_4|+1,\ldots, |J_{4,5}|\}$ such that
$$
[\cG_{j_{\iota}}P_{\iota-1}]
T_{\iota-1}^{-1}  [\cG_{j_{\iota}}P_{\iota-1}]^\top
\le 0.
$$
Let $k:=\iota-1$ in \eqref{orderedproduction}.
According to \eqref{coderzero} and Part (6e) of Lemma \ref{codersecond}, one can fix
all $p^{j_i}=0$ and all $\bfp^{j_i}=\bfzero$ in \eqref{orderedproduction} for all $i\neq \iota$.
By doing this one can obtain that
\[
\label{reduced5}
T_{\iota-1}
\bfeta
\neq
-P_{\iota-1}^{\top}
\cG_{j_{\iota}}^\top p,
\quad
\forall\, p \in \cD^*\cN_{[0,+\infty)}(0,0) (\cG_{j_{\iota}}P_{\iota-1}\bfeta),
\quad
\forall\, \bfeta \in\Re^{\ell-\iota +1}
\setminus\{\bfzero\}.
\]
From \eqref{defginew} one has $\cG_{j_{\iota}}=(\dot g^{j_{\iota}}(\bfx^*);-\bar g^{j_{\iota}}(\bfx^*))^\top
\cJ g^{j_{\iota}}(\bfx^*)H$.
Then, by letting $\bfeta:=T_{\iota -1}^{-1}
[\cG_{j_{\iota}} P_{\iota-1}]^\top$ one can get
\[
\label{qlezero5}
(\dot g^{j_{\iota}}(\bfx^*);-\bar g^{j_{\iota}}(\bfx^*))^\top
\cJ g^{j_{\iota}}(\bfx^*)H P_{{\iota}-1}\bfeta
=
\cG_{j_{\iota}}P_{{\iota}-1}\bfeta
=
[\cG_{j_{\iota}}P_{{\iota}-1}] T_{{\iota}-1}^{-1}  [\cG_{j_{\iota}}P_{{\iota}-1}]^\top
\le 0.
\]
Thus, it can be easily observed from \eqref{coderzero} that
$p:=
-1
\in \cD^*\cN_{[0,+\infty)}(0,0)(\cG_{j_{\iota}}P_{{\iota}-1}\bfeta)$.
Moreover, by the definition of $\bfeta$ we have
$T_{{\iota}-1}\bfeta=[\cG_{j_{\iota}}P_{{\iota}-1}]^\top=P_{{\iota}-1}^\top \cG_{j_{\iota}}^\top
=-P_{{\iota}-1}^\top \cG_{j_{\iota}}^\top p$.
However, this contradicts \eqref{reduced5}, so that
\eqref{qlezero5} fails. Therefore, \eqref{positivequad} holds for all $i\in\{|J_4|+1,\ldots, |J_{4,5}|\}$.

\smallskip
{\bf(3)}
Suppose that there exists an index $\iota\in\{|J_{4,5}|+1,\ldots, |J_{4,5,6}|\}$ such that
\[
\label{npsdass}
[\cG_{j_{\iota}}P_{{\iota}-1}] T_{{\iota}-1}^{-1}  [\cG_{j_{\iota}}P_{{\iota}-1}]^\top \not\succ O.
\]
Let $k:=\iota-1$ in \eqref{orderedproduction2}.
According to Part (6e) of Lemma \ref{codersecond}
(or \eqref{coderzero} if $r_{j_{\iota}=0}$),
one can fix $\bfp^{j_i}=\bfzero$ in \eqref{orderedproduction2} for all $i\neq \iota$.
By doing this and using the fact that $\tilde k= k+\sum_{t=|J_{4,5}|+1}^k r_{j_t}$
one can obtain that
\[
\label{reduced6}
T_{\iota-1} \bfeta
\neq
-P_{\iota-1}^\top
\cG_{j_{\iota}}^\top
\bfp,
\quad \forall \,
\bfp \in \cD^*\cN_{\cQ_{j_{\iota}}}(\bfzero,\bfzero)(\cG_{j_{\iota}} P_{\iota-1}\bfeta ),
\quad
\forall\,  \bfeta \in \Re^{\ell-\tilde k} \setminus\{\bfzero\}.
\]
Note that $\ker P_{\iota-1}^\top=\Span
\big(\{\cG_{j_{i}}^\top\}_{i=1,\ldots, \iota-1}\big)$.
Moreover, recall from \eqref{rrlation} and \eqref{defp} that each $P_i$ has full column rank.
Suppose that $[\cG_{j_{\iota}}P_{{\iota}-1}] T_{{\iota}-1}^{-1}  [\cG_{j_{\iota}}P_{{\iota}-1}]^\top$ is singular, i.e., there exists a nonzero vector $\hat\bftheta\in\Re^{1+r_{j_{\iota}}}$ such that
$[\cG_{j_{\iota}}P_{{\iota}-1}]
T_{{\iota}-1}^{-1}
[\cG_{j_{\iota}}P_{{\iota}-1}]^\top \hat\bftheta=\bfzero.$
Then, on the one hand, if $\hat\bfeta:=T_{{\iota}-1}^{-1} [\cG_{j_{\iota}}P_{{\iota}-1}]^\top \hat\bftheta=\bfzero$,
one has
$P^\top_{{\iota}-1}\cG_{j_{\iota}}^\top\hat\bftheta=
[\cG_{j_{\iota}}P_{{\iota}-1}]^\top \hat\bftheta=\bfzero$,
so that
$$\cG_{j_{\iota}}^\top\hat\bftheta\in \Span\big(\{\cG_{j_{i}}^\top\}_{i=1,\ldots, \iota-1}\big),
$$
but this contradicts \eqref{fullrange} of Corollary \ref{aubins2}.
On the other hand, if  $\hat\bfeta\neq\bfzero$,
one has from the fact $\cG_{j_{\iota}} P_{\iota-1}\hat\bfeta =\bfzero$ and Part (6a) of Lemma \ref{codersecond}
(or \eqref{coderzero} if $r_j=0$) that
$-\hat\bftheta \in
\cD^*\cN_{\cQ_{j_{\iota}}}(\bfzero,\bfzero)
(\cG_{j_{\iota}} P_{\iota-1}\hat \bfeta)$.
Moreover, one can see that  $T_{\iota-1} \hat \bfeta
=[\cG_{j_{\iota}}P_{{\iota}-1}]^\top \hat\bftheta
=-P_{\iota-1}^\top
\cG_{j_{\iota}}^\top (-\hat\bftheta)$, which contradicts \eqref{reduced6}.
Therefore, $[\cG_{j_{\iota}}P_{{\iota}-1}] T_{{\iota}-1}^{-1}  [\cG_{j_{\iota}}P_{{\iota}-1}]^\top$
is nonsingular.
Then, one can define the matrix $M:=([\cG_{j_{\iota}}P_{{\iota}-1}]T_{{\iota}-1}^{-1} [\cG_{j_{\iota}}P_{{\iota}-1}]^\top)^{-1}$.

For a vector $\bfq\in\Re^{1+r_{j_{\iota}}}$, one can define
$\bfp:=- M\bfq$ and $\bfeta:=-T_{{\iota}-1}^{-1} [\cG_{j_{\iota}}P_{{\iota}-1}]^\top \bfp$.
Then one has $\cG_{j_{\iota}}P_{{\iota}-1}\bfeta=\bfq$ and, consequently,
$T_{{\iota}-1} \bfeta
=
-[\cG_{j_{\iota}}P_{{\iota}-1}]^\top \bfp
=
-P_{{\iota}-1}^\top \cG_{j_{\iota}}^\top  \bfp$.
Moreover, since \eqref{reduced6} holds, one obtains that
\[
\label{pqfails}
-M\bfq \notin \cD^*\cN_{\cQ_{j_{\iota}}}(\bfzero,\bfzero)(\bfq),\quad\forall \bfq\in\Re^{1+r_{{j}_\iota}}.
\]

Next, we show that \eqref{npsdass} can not be true.
We first consider the case that $r_{j_{\iota}}=0$.
If \eqref{npsdass} holds in this case,
by setting $q=[\cG_{j_{\iota}}P_{{\iota}-1}] T_{{\iota}-1}^{-1}  [\cG_{j_{\iota}}P_{{\iota}-1}]^\top< 0$,
one has $-M^{-1}q=-1 < 0$.
Then, by \eqref{coderzero} we know that  $-Mq\in \cD^*\cN_{\cQ_{j_{\iota}}}(0,0)(q)$, which contradicts
\eqref{pqfails}.
Consequently, one has  $[\cG_{j_{\iota}}P_{{\iota}-1}] T_{{\iota}-1}^{-1}  [\cG_{j_{\iota}}P_{{\iota}-1}]^\top>0$ in this case.
In the following, we deal with the case that $r_{j_{\iota}}\ge 1$.
Firstly, we shall prove  that
\[
\label{pdcond1}
\langle \bftheta,  M^{-1}\bftheta\rangle> 0,
\quad\forall\, \bftheta\in
\boundary\cQ_{{j_{\iota}}}\setminus\{\bfzero\}.
\]
Suppose on the contrary that there exists  a vector $\hat\bftheta\in\boundary\cQ_{{j_{\iota}}}\setminus\{\bfzero\}$ such that
$
\langle \hat\bftheta, M^{-1}\hat\bftheta\rangle\le 0.
$
By taking $\hat\bfq:=M^{-1}\bf\hat\bftheta$, one has $\langle -\hat \bftheta,\hat \bfq\rangle \ge 0$.
Then, by Part (6c) of Lemma \ref{codersecond} we know that
$-M\hat \bfq=-\hat\bftheta \in  \cD^*\cN_{\cQ_{j_{\iota}}}(\bfzero,\bfzero)(\hat\bfq)$, which contradicts \eqref{pqfails}.
Therefore, \eqref{pdcond1} holds.
Secondly, we shall show that the following holds:
\[
\label{pdcond2}
\langle
\bftheta,  M^{-1}\bftheta\rangle>0,
\quad\forall\, \bftheta\in
\interior\cQ_{j_{\iota}} .
\]
For the purpose of contradiction, assume that \eqref{pdcond2} does not hold.
Since the matrix $ M^{-1}$ is nonsingular and symmetric,
we know from \eqref{pdcond1} and Lemma \ref{lemtech} that
there exists a nonzero vector
$\tilde \bfp\in -\interior\cQ_{j_{\iota}}$ such that $\tilde \bfq:=- M^{-1}\tilde \bfp\in -\cQ_{j_{\iota}}$.
From Part (6b) of Lemma \ref{codersecond}
one can see that
$-M\tilde  \bfq= \tilde \bfp \in  \cD^*\cN_{\cQ_{j_{\iota}}}(\bfzero,\bfzero)(\tilde\bfq)$, which contradicts \eqref{pqfails}.
Therefore, \eqref{pdcond2} holds.
Thirdly, we shall show that
\[
\label{pdcond2+}
\langle
\bfxi, M \bfxi\rangle>0,
\quad\forall\,  \bfxi\in \cQ_{j_{\iota}}^\circ\setminus\{\bfzero\}.
\]
Since $\cQ_{j_{\iota}}$ is a self-dual cone, one has $\cQ_{j_{\iota}}^\circ=-\cQ_{j_{\iota}}$.
On the one hand, if there exists a nonzero vector $\hat\bfxi\in -\boundary\cQ_{j_{\iota}}$ such that
$\langle\hat \bfxi, M \hat \bfxi\rangle\le 0$, then we know from Part (6d) of Lemma \ref{codersecond} that
$-M \bfxi
\in \cD^*\cN_{\cQ_{j_{\iota}}}(\bfzero,\bfzero)(\bfxi)$,
which contradicts \eqref{pqfails}. Therefore, it holds that
\[
\label{pdcond2++}
\langle
\bfxi, M \bfxi\rangle>0,
\quad\forall\,  \bfxi\in \boundary\cQ_{j_{\iota}}^\circ\setminus\{\bfzero\}.
\]
On the other hand,  if there exists a vector $\tilde\bfxi\in -\interior \cQ_{j_{\iota}}$ such that
$\langle \tilde\bfxi, M \tilde\bfxi\rangle\le 0$, we obtain
from \eqref{pdcond2++} and Lemma \ref{lemtech} that
 there exists a nonzero vector $\bfv\in -\interior \cQ_{j_{\iota}}$
such that $M \bfv\in \cQ_{j_{\iota}}$.
Then, one has $\langle \bfv, M  \bfv\rangle \le 0$,
i.e.,  $\langle (M\bfv), M^{-1} (M \bfv)\rangle \le 0$.
This contradicts either \eqref{pdcond1} or \eqref{pdcond2} since $\bfzero\neq M \bfv\in \cQ_{j_{\iota}}$. Thus  \eqref{pdcond2+} holds.
Combining \eqref{pdcond1}, \eqref{pdcond2}, and \eqref{pdcond2+} we know from Lemma \ref{lemma:pd} that
\eqref{npsdass} does not hold.
This completes the proof.
\end{proof}

Based on the previous discussions  in this section,
we can establish the following result on the strong second-order sufficient condition.

\begin{proposition}
\label{propmain}
Let $\bfx^*$ be a locally optimal solution to the nonlinear SOCP \eqref{nlsocp} with
$\bfy^*=(\bflambda^*,\bfmu^*)$ being an associated   multiplier.
 {Let the index set $\{1,\ldots,J\}$ be classified as in \eqref{jdecompose}.}
Assume that the solution mapping ${\rS}_{\rm KKT}$ given in \eqref{skkt} has the Aubin property at $(\bfzero,\bfzero)$ for $(\bfx^*,\bfy^*)$.
Then the strong second-order sufficient condition \eqref{ssosc} holds for \eqref{nlsocp} at $\bfx^*$.
\end{proposition}
\begin{proof}

Since the solution mapping ${\rS}_{\rm KKT}$ in \eqref{skkt} has the Aubin property at $(\bfzero,\bfzero)$ for $(\bfx^*,\bfy^*)$, we know that the mapping $\rS_{\rm GE}$ in \eqref{sge} also has the Aubin property at the origin for $\bfx^*$.
Then, by \cite[Corollary 25]{ding2017} the constraint nondegeneracy condition \eqref{cnd} holds at $\bfx^*$ and the second-order sufficient condition holds \eqref{sosc} at $\bfx^*$.
Therefore, Assumption \ref{assblanket} holds.

Let $\Theta\in\Re^{n\times n}$  be the matrix defined in \eqref{ssocq},
and  $H\in\Re^{n\times \ell}$ be a matrix with full column rank such that $\range H=\aff(\cC(\bfx^*))$.
From Corollary \ref{aubinss} we know that \eqref{aubrxnew} holds with $\cG_j$ ($j\in J_{4,5,6}$) being defined by \eqref{defgj}.
Furthermore, one knows that \eqref{fullrange} holds
by Corollary \ref{aubins2}.
Then, based on \eqref{jirule} and \eqref{defginew},
one has \eqref{positivequad} holds for the matrices $R_i$, $P_i$ and $T_i$ defined in \eqref{defr}, \eqref{defp} and \eqref{defti} by Proposition \ref{propreduce}.

From Lemma \ref{lemsosc} we know that $\langle \bfnu, H^\top \Theta H \bfnu \rangle>0, \, \forall  \bfnu\in\cV\setminus\{\bfzero\}$,
where $\cV$ is defined in \eqref{vg}.
Therefore, one has
\[
\label{soscfinal}
\langle \bfnu, H^\top \Theta H \bfnu \rangle>0, \  \forall \, \bfzero\neq\bfnu
\in
\lin\cV
=\left\{ \bfnu\in\Re^\ell \, \middle| \,
\begin{array}{ll}
\cG_{j_i} \bfnu= 0,
&
1\le i\le |J_{4,5}|, \\
\cG_{j_i} \bfnu =\bfzero, &  |J_{4,5}|\le i\le |J_{4,5,6}|
\end{array}
\right\}.
\]
On the one hand, if $\bfnu\in\lin\cV$,
one has $\cG_{j_1}\bfnu=0$.
From \eqref{rrlation} one has
$\range\, R_1=\ker\, (\cG_{j_1} P_{0})=\ker\, \cG_{j_1}$,
so that $\bfnu\in \range R_1=\range P_1$.
Furthermore, if one has  $\bfnu\in \range P_{i-1}$ for a certain $i\in\{2,\ldots, |J_{4,5,6}|\}$,
one can get from $\eqref{rrlation}$ that $\range\, R_i=\ker\, \cG_{j_i} P_{i-1}$.
Then from the definition of $\lin\cV$ in \eqref{soscfinal} one can see that $\bfnu\in\range P_{i-1} R_{i}=\range P_i$.
Consequently, one can get by induction that
\[
\label{vinrangepmax}
\bfnu\in \range P_{|J_{4,5,6}|}.
\]
On the other hand, if \eqref{vinrangepmax} holds, then
it is easy to see from \eqref{rrlation} that $\bfnu\in \lin\cV$ holds.
Therefore, $\bfnu\in \lin\cV$ if and only if \eqref{vinrangepmax} holds.
Consequently, by \eqref{soscfinal} one has
$\langle \bfnu, H^\top \Theta H \bfnu \rangle>0$ for all $\bfzero \neq \bfnu\in \range P_{|J_{4,5,6}|}$.
This is equivalent, by \eqref{defp} and \eqref{defti}, to the condition that
\[
\label{timaxge}
T_{i}=
P_{i}^\top H^\top \Theta H P_{i}
=R_{i}^\top T_{i-1} R_{i}
\succ O
\quad\mbox{with}\quad
i=|J_{4,5,6}|.
\]
Recall that each $R_i$ defined in \eqref{defr} has full column rank and each
$P_i$ in \eqref{defp} also has full column rank.
From Corollary \ref{aubins2} we know that each $\cG_{j_i}$ is surjective.
Then by \eqref{rrlation} we know that
\[
\label{rrlationx}
(\range\, R_i)^\circ=\range [\cG_{j_i} P_{i-1}]^\top,
\quad\forall\, i=1,\ldots, |J_{4,5,6}|.
\]
Recall from \eqref{positivequad} of Proposition \ref{propreduce} that
\[
\label{positivequadx}
[\cG_{{j_{i}}} P_{i-1}] T_{i-1}^{-1}
[\cG_{j_{i}}P_{i-1}]^\top \succ O
\quad\mbox{with}\quad
i=|J_{4,5,6}|.
\]
Based on \eqref{timaxge}, \eqref{rrlationx} and \eqref{positivequadx}, one can apply Lemma \ref{lemma:pd} to get
$T_{|J_{4,5,6}|-1}\succ O$.
Furthermore, for any index $i\in\{2,\ldots, |J_{4,5,6}|-1\}$, if one has
\[
\label{succplus}
T_{i}=R_i^\top T_{i-1} R_i\succ O,
\]
one can get
from \eqref{positivequad} of Proposition \ref{propreduce}
that
$ [\cG_{j_i}P_{i-1}] T_{i-1}^{-1}  [\cG_{j_i}P_{i-1}]^\top \succ O$, which, together with \eqref{rrlationx} and \eqref{succplus}, implies $T_{i-1}\succ O$
by Lemma \ref{lemma:pd}.
Therefore, one can get $T_{0}=H^\top \Theta H\succ O$ by induction.
Note that $\range H=\aff(\cC(\bfx^*))$
and $\Theta\in\Re^{n\times n}$ is the matrix defined in \eqref{ssocq}. Then from Lemma \ref{lemsosc} we know that the strong second-order sufficient condition \eqref{ssosc} holds.
\end{proof}

\subsection{The equivalence theorem}
The main result of this paper is given in the following theorem.
\begin{theorem}
\label{thm:main}
Let $\bfx^*$ be a locally optimal solution to the nonlinear SOCP \eqref{nlsocp} and
$\bfy^*=(\bflambda^*,\bfmu^*)$ be a corresponding  multiplier.
The following two statements are equivalent:

\begin{itemize}
\item[\rm \bf (1)]
the solution mapping ${\rS}_{\rm KKT}$ in \eqref{skkt} has the Aubin property at $(\bfzero,\bfzero)$ for $(\bfx^*,\bfy^*)$;

\item[\rm \bf (2)]
the solution $(\bfx^*,\bfy^*)$ is   strongly regular   to the KKT system \eqref{ge}.
\end{itemize}
\end{theorem}

\begin{proof}
Note that ({\bf 1}) is an immediate consequence of ({\bf 2}).
On the other hand, if ({\bf 1}) holds,
we know from \cite[Corollary 25]{ding2017} that
$\bfx^*$ is nondegenerate and the second-order sufficient condition \eqref{sosc} holds at $\bfx^*$.
Consequently, it follows from  Proposition \ref{propmain} that the strong second-order sufficient condition \eqref{ssosc} holds.  Thus, by using  \cite[Theorem 30]{bonnans2005} we know that ({\bf 2}) holds.
This completes the proof.
\end{proof}

Theorem \ref{thm:main}  establishes the equivalence between the Aubin property and the strong regularity for the nonlinear SOCP \eqref{nlsocp} at local optimal solutions without requiring  {the} strict complementarity.
Moreover, Theorem \ref{thm:main} implies its counterpart for conventional nonlinear programming $($\cite[Theorems 1, 4 $\&$ 5]{don1996}$)$ as a special case as one can take $r_j=0$ for all $j=1,\dots, J$.

According to Theorem \ref{thm:main}, ${\rS}_{\rm KKT}$ in \eqref{skkt} having the Aubin property at $(\bfzero,\bfzero)$ for $(\bfx^*,\bfy^*)$ is also equivalent to many other conditions, such as
(a) $\bfx^*$ is nondegenerate and the uniform second order growth condition (\cite[Definition 5.16]{B&S2000}) holds at $\bfx^*$ (by \cite[Theorem 30]{bonnans2005}); (b)
$\bfx^*$ is a nondegenerate and fully
stable (\cite[Definition 4.1]{fullstab}) locally optimal solution (by \cite[Theorem 4.8]{fullstab}); and
(c) Clarke's generalized Jacobian of a KKT system is nonsingular (by \cite[Theorem 3.1]{w&z2009}).  One may refer to the relevant references for details.

\section{Conclusions}
\label{sec:conclu}
In this paper, we established the equivalence between the Aubin property of the perturbed KKT system and the strong regularity of the KKT system for nonlinear SOCP problems at locally optimal solutions. 
Our results extend beyond prior work by removing the restrictive condition $|J_4 \cup J_5 \cup J_6| \le 1$ required in \cite{outrata2011} and \cite{opazo2017}.
We achieved this by introducing a novel reduction approach to derive the strong second-order sufficient condition from the Aubin property of the perturbed KKT system, in which the lemma of alternative choices on cones we developed here plays an essential role.
Our findings prompt further investigations into whether the approach introduced here could lead to similar equivalences for a broader range of (conic) optimization problems.
Recently,  for nonlinear semidefinite programming (SDP), \cite{ccsz2024} achieved the same equivalence after the announcement of this work, but their approach leverages properties specific to nonlinear SDP. 
While SOCPs can be reformulated as SDPs, simultaneous satisfaction of constraint nondegeneracy can fail \cite{zhao,B&S2000}, implying that the results in \cite{ccsz2024} cannot   carry over to SOCPs\footnote{
To see this, one may consider the nonlinear SOCP 
of minimizing 
$f(x_0,x_1):= \frac{1}{2}(x_0^2+2x_0+x^2_1)$ subject to $(x_0, x_1)\in \{(x_0,x_1) \mid x_0\ge |x_1| \}$, together with its equivalent SDP reformulation of minimizing $f(x_0,x_1)$ subject to $
\begin{pmatrix}
x_0&x_1\\x_1&x_0
\end{pmatrix}\succeq 0$,
as an illustrative example.}.
The equivalence between the Aubin property and the strong regularity for general optimization problems, particularly those with non-polyhedral $C^2$-cone reducible constraints, remains an open question.  
We leave this as one of our future research topics.

\section*{Acknowledgments}
The authors would like to thank the two anonymous referees and the associate editor for their many valuable comments and constructive
suggestions including the references mentioned in their reports, 
which have substantially helped improve the quality and presentation of this paper. 

\small
\bibliographystyle{plain}

\end{document}